\documentclass[reqno,11pt,letterpaper]{amsart}
\usepackage{amsmath,amssymb,amsthm,graphicx,mathrsfs,url,dsfont,enumitem}
\usepackage[usenames,dvipsnames]{color}
\usepackage[colorlinks=true,linkcolor=Red,citecolor=Green,urlcolor=Blue]{hyperref}
\usepackage{dsfont}
\usepackage[super]{nth}
\usepackage[open, openlevel=2, depth=3, atend]{bookmark}
\hypersetup{pdfstartview=XYZ}

\def\?[#1]{\textbf{[#1]}\marginpar{\Large{\textbf{??}}}}

\numberwithin{equation}{section}

\newtheorem{theorem}{Theorem}[section]
\newtheorem{lemma}[theorem]{Lemma}
\newtheorem{proposition}[theorem]{Proposition}

\newtheorem{remark}[theorem]{Remark}
\newtheorem{definition}[theorem]{Definition}

\newcommand{\mc}{\mathcal}

\newcommand{\norm}[1]{\left\Vert#1\right\Vert}

\newcommand{\bbar}{\overline}

\let\Im=\Imag

\let\Re=\Real

\DeclareMathOperator{\Vol}{Vol}

\renewcommand{\tilde}{\widetilde}          
\DeclareMathSymbol{\leqslant}{\mathalpha}{AMSa}{"36} 
\DeclareMathSymbol{\geqslant}{\mathalpha}{AMSa}{"3E} 
\DeclareMathSymbol{\eset}{\mathalpha}{AMSb}{"3F}     
\renewcommand{\leq}{\;\leqslant\;}                   
\renewcommand{\geq}{\;\geqslant\;}                   

\newcommand{\C}{\mathbb{C}}
\newcommand{\D}{\mathbb{D}}
\newcommand{\R}{\mathbb{R}}

\renewcommand{\H}{\mathbb{H}}
\newcommand{\N}{\mathbb{N}}
\newcommand{\Q}{\mathbb{Q}}
\newcommand{\U}{\mathbb{U}}
\newcommand{\E}{\mathbb{E}}
\renewcommand{\P}{\mathbb{P}}

\def\S{\mathbb{S}}

\def\bi{\begin{itemize}}
\def\ei{\end{itemize}}

\def\bnum{\begin{enumerate}}
\def\enum{\end{enumerate}}
\def\<#1{\langle #1 \rangle}

\def\t{\mathbf{t}}

\def\z{\mathbf{z}}

\def\U{\Upsilon_{\frac{\gamma}{2}}}

\def\cyl{\mathcal{C}_\infty}

\def\x{\mathbf{x}}

\title[Fusion in LCFT]{Fusion asymptotics for Liouville correlation functions}

\author[Guillaume Baverez]{Guillaume Baverez$^1$}
\thanks{$^1$Supported by the EPSRC grant EP/L016516/1 for the University of Cambridge CDT (CCA)}
 \address{University of Cambridge, Centre for Mathematical Sciences, Cambridge CB3 0WA, UK}
 \email{gb539@cam.ac.uk}
 
 \author[Mo Dick Wong]{Mo Dick Wong$^{1,2}$}
\thanks{$^2$Supported by a Croucher Foundation Scholarship 
}
 \address{University of Cambridge, Centre for Mathematical Sciences, Cambridge CB3 0WA, UK}
 \email{mdw46@cam.ac.uk}

\begin{document}

\maketitle
 
\begin{abstract}
In \cite{DKRV1}, David-Kupiainen-Rhodes-Vargas introduced a probabilistic framework based on the Gaussian Free Field and Gaussian Multiplicative Chaos in order to make sense rigorously of the path integral approach to Liouville Conformal Field Theory (LCFT). We use this setting to compute fusion estimates for the four-point correlation function on the Riemann sphere, and find that it is consistent with predictions from the framework of theoretical physics known as the conformal bootstrap. This result fits naturally into the famous KPZ conjecture \cite{KPZ} which relates the four-point function to the expected density of points around the root of a large random planar map weighted by some statistical mechanics model.

From a purely probabilistic point of view, we give non-trivial results on negative moments of GMC. We give exact formulae based on the DOZZ formula in the Liouville case and asymptotic behaviours in the other cases, with a probabilistic representation of the limit.

Finally, we show how to extend our results to boundary LCFT, treating the cases of the fusion of two boundary or bulk insertions as well as the absorption of a bulk insertion on the boundary. 
%
\end{abstract}

\footnotesize



\normalsize



\tableofcontents

 
\section{Introduction}
\label{sec:intro}
	\subsection{Path integral}
	\label{subsec:path_integral}
The Liouville action on the Riemann sphere $\S^2\cong\widehat{\C}=\C\cup\{\infty\}$ is the action functional $S_L:\Sigma\to\R$ (where $\Sigma$ is some function space to be determined) defined by\footnote{Usually the Liouville action has also a curvature term, which we have omitted here for simplicity. This will not play an important role since we will consider metrics whose curvature concentrates on the unit circle.}
\begin{equation}
S_\mathrm{L}(X)=\frac{1}{4\pi}\int_{\S^2}(|\nabla X|^2+4\pi\mu e^{\gamma X}g(z))d^2z
\end{equation}
where $g(z)=|z|_+^{-4}:=(|z|\vee 1)^{-4}$ is the background metric, $\gamma\in(0,2)$ is the parameter of the theory, and $\mu>0$ is the cosmological constant (whose value is irrelevant in this paper). Another important parameter is the so-called \emph{background charge} which is defined by $Q:=\frac{\gamma}{2}+\frac{2}{\gamma}$. From here, Liouville Conformal Field Theory (LCFT) is the ``Gibbs measure" associated to $S_L$, which is formally defined in the physics literature by
\begin{equation}
\label{eq:lcft_path_integral}
\langle F\rangle:=\int F(X)e^{-S_\mathrm{L}(X)}DX
\end{equation}
for all continuous functional $F$ on $\Sigma$. Here $DX$ stands for ``Lebesgue measure" on $\mc{C}^\infty(\S^2)$, which of course does not make sense mathematically. Nonetheless, it is possible to define \eqref{eq:lcft_path_integral} in a rigorous framework using the Gaussian Free Field (GFF) and Kahane's theory of Gaussian Multiplicative Chaos (GMC) -- see \cite{DKRV1} and sections \ref{subsec:gff} and \ref{subsec:gmc} of this paper.
Roughly speaking, the GFF $X$ on $\S^2$ is the Gaussian field corresponding to the ``Gaussian measure" $e^{-\frac{1}{4\pi}\int_{\S^2}|\nabla X|^2}DX$. We will write $\P$ for its probability measure and $\E$ for the associated expectation. The GFF lives $\P$-a.s. in the topological dual of the Sobolev space $H^1(\S^2,g)$ and is therefore defined as a distribution (in the sense of Schwartz). In this context, GMC is the random measure $M^\gamma$ on $\S^2$ defined for all $\gamma\in(0,2)$ and making sense of the exponential of the GFF (which is \textit{a priori} ill-defined). This can be constructed through a regularisation of the field and we will loosely write $dM^\gamma(z)=e^{\gamma X(z)-\frac{\gamma^2}{2}\E[X(z)^2]}g(z)d^2z$ to refer to the limiting measure, even though $X$ is only defined as a distribution.

The main observables in LCFT are the \emph{vertex operators} $V_\alpha(z):=e^{\alpha X(z)}$, giving rise to the \emph{correlation functions}, which can be thought of as the Laplace transform of the field defined by the measure \eqref{eq:lcft_path_integral}:
\begin{equation}
\left\langle\prod_{i=1}^NV_{\alpha_i}(z_i)\right\rangle=\int\prod_{i=1}^Ne^{\alpha_i X(z_i)}e^{-S_\mathrm{L}(X)}DX
\end{equation}
On the sphere, these are defined for all pairwise disjoint \emph{insertions} $(z_1,...,z_n)\in\widehat{\C}^N$ and \emph{Liouville momenta} $(\alpha_1,...,\alpha_n)\in\R_+^N$ satisfying the \emph{Seiberg bounds}
\begin{equation}
\label{eq:seiberg_bounds}
\sigma:=\sum_{i=1}^N\frac{\alpha_i}{Q}-2>0\qquad\qquad\forall i,\;\frac{\alpha_i}{Q}<1
\end{equation}
In particular, this implies that the correlation function exists only if $N\geq3$.

For fixed $z_0\in\widehat{\C}$, the vertex operator $V_\alpha(z_0)$ has a geometric interpretation, as it inserts a conical singularity of order $\alpha/Q$ at $z_0$ in the physical metric (\cite{Sei,HMW}, Appendix \ref{app:conical}). Thus the second Seiberg bound is there to make the singularity integrable around $z_0$. On the other hand by Gauss-Bonnet theorem, the first bound is equivalent to asking that the surface $\S^2\setminus\{z_1,...,z_N\}$ with conical singularities of order $\alpha_i/Q$ at $z_i$ has negative total curvature.

 The correlation functions satisfy some conformal covariance under M\"obius transformation, namely if $\psi$ is such a map, then \cite{DKRV1} 
 \[\left\langle\prod_{i=1}^NV_{\alpha_i}(\psi(z_i))\right\rangle=\prod_{i=1}^N|\psi'(z_i)|^{-2\Delta_i}\left\langle\prod_{i=1}^NV_{\alpha_i}(z_i)\right\rangle\] 
where $\Delta_i=\Delta_{\alpha_i}:=\frac{\alpha_i}{2}(Q-\frac{\alpha_i}{2})$ is called the \emph{conformal dimension} of $V_{\alpha_i}(\cdot)$. This property implies that the three-point correlation function $\langle\prod_{i=1}^3V_{\alpha_i}(z_i)\rangle$ is determined by $\langle V_{\alpha_1}(0)V_{\alpha_2}(1)_{\alpha_3}(\infty)\rangle$ since there is a unique M\"obius transformation sending $(z_1,z_2,z_3)$ to $(0,1,\infty)$. The three-point correlation functions play a central role in the conformal bootstrap approach to CFTs (see Section \ref{subsec:conformal_bootstrap}). For LCFT, they are given by the celebrated DOZZ formula, a proof of which was given for the first time in \cite{KRV2}, where the authors rigorously implemented the method known as Teschner's trick \cite{T95} (see \cite{DO,ZZ} for the original derivation of the formula which uses a different approach).

We now turn to the four-point function. By conformal covariance, we can take the insertions to be at $(z_1,z_2,z_3,z_4)=(0,z,1,\infty)$ with $z\in\widehat{\C}\setminus\{0,1,\infty\}$ being the free parameter. In this paper, we will take $(\alpha_1,\alpha_2,\alpha_3,\alpha_4)$ satisfying the Seiberg bounds and will be concerned about the behaviour of the four-point function as $z\to0$ (the other fusions being easily deduced from conformal invariance). In the framework of \cite{DKRV1} using the GFF and GMC, the four-point function has the following expression for $|z|\leq1$:
\begin{equation}
\label{eq:def_correl_intro}
\begin{aligned}
\left\langle\prod_{i=1}^4V_{\alpha_i}(z_i)\right\rangle=2\gamma^{-1}\mu^{-\frac{Q\sigma}{\gamma}}\Gamma\left(\frac{Q\sigma}{\gamma}\right)|z|^{-\alpha_1\alpha_2}|1-z|^{-\alpha_2\alpha_3}\E\left[\left(\int_{\widehat{\C}}e^{\gamma\sum_{i=1}^4\alpha_iG(z_i,\cdot)}dM^\gamma\right)^{-\frac{Q\sigma}{\gamma}}\right]
\end{aligned}
\end{equation}
where $G=G(\cdot,\cdot)$ is Green's function on $(\S^2,g)$. The main feature of \eqref{eq:def_correl_intro} is that, up to explicit factors, it is expressed using negative moments of GMC. One of our main results (Theorem \ref{thm:result}) gives the exact asymptotic behaviour of \eqref{eq:def_correl_intro} as $z\to0$ using the integrability result of the DOZZ formula. Now the reader will notice that the negative exponent in the definition of \eqref{eq:def_correl_intro} depends on the $\alpha_i$'s, so the DOZZ formula does not give integrability results for \emph{all} moments of GMC but only for the one corresponding to the Liouville correlation function. However, in our framework, we lose nothing in promoting $\sigma$ to a free parameter, so we were able to find the asymptotic behaviour of all negative moments (Theorems \ref{theo:result_gmc} and \ref{theo:prob_rep1}) but only in the Liouville case did we get an exact expression for the limit. In this special case, we were able to confirm a prediction coming from the bootstrap approach to LCFT, which we review now.

	\subsection{Conformal bootstrap}
	\label{subsec:conformal_bootstrap}
The foundations of the conformal bootstrap were laid in \cite{BPZ} and since then it has been acknowledged in the physics community as a powerful tool to analyse two dimensional CFTs. However it is still a challenge to make sense of the theory in a rigorous mathematical framework. One of the goals of this paper is to recover some aspects of the bootstrap predictions in the probabilistic formulation of LCFT. 

The conformal bootstrap is an algebraic approach based on the axiom that the vertex operator $V_\alpha$ can be associated to a highest-weight representation of the Virasoro algebra \cite{Ri}. It turns out that this assumption constrains the correlation functions drastically through some identities like the the Ward or BPZ equations (a null-vector equation at level 2). 
The constraints of local conformal invariance imply that all correlation functions can be constructed from more fundamental objects:
\begin{enumerate}
\item The \emph{spectrum} $\mc{S}\subset\C$. For $\alpha\in\mc{S}$, the vertex operator $V_\alpha(\cdot)$ is called a \emph{primary field}. In Liouville CFT, the spectrum is the line $Q+i\R$. It is important to notice that the conformal bootstrap assumes that vertex operators are defined for all $\alpha\in\C$ and not necessarily for $\alpha$ in the ``physical region" defined by the Seiberg bounds.
\item The 3-point correlation functions, a.k.a. the \emph{structure constants}. In Liouville CFT, these are given by the DOZZ formula $C_\gamma(\alpha_1,\alpha_2,\alpha_3)$, which is meromorphic in each one of the $\alpha_i$'s.
\end{enumerate}
Another key idea of the conformal bootstrap is that local fields should satisfy a so-called \emph{Operator Product Expansion} (OPE), which can be understood analytically as a Taylor expansion of vertex operators in the $z$ variable. In other words, the OPE of the local operators $V_{\alpha_1}(0)V_{\alpha_2}(z)$ describes the fusion of the two insertions as $z\to0$. The fusion rule is particularly simple in the case where the Verma module associated to $V_{\alpha_2}(z)$ is reducible (i.e. $\alpha_2\in-\frac{\gamma}{2}\N^*-\frac{2}{\gamma}\N^*$), but in the case of $\alpha_1,\alpha_2$ in the spectrum, it has the following form (\cite{BZ}, equation (1.18))
\begin{equation}
\label{eq:ope}
V_{\alpha_1}(z)V_{\alpha_2}(0)=\frac{1}{8\pi}\int_\R|z|^{2(\Delta_P-\Delta_1-\Delta_2)}C_\gamma(\alpha_1,\alpha_2,Q-iP)V_{Q+iP}(0)|f_{\gamma,P}^{\alpha_{12}}(z)|^2dP
\end{equation}
where $\Delta_P=\frac{Q^2}{4}+\frac{P^2}{4}$ is the conformal dimension of $V_{Q-iP}$ and $f_{\gamma,P}^{\alpha_{12}}(z)=1+o_{z\to0}(1)$ is a so-called \emph{conformal block}, a holomorphic function of $z$ depending only on $P,\gamma,\alpha_1,\alpha_2$. Plugging this into the four-point correlation function yields\footnote{We add the superscript $^{\mathrm{cb}}$ for ``conformal bootstrap", in order to differentiate it with the correlation function given by the path integral.}
\begin{equation}
\label{eq:bootstrap}
\begin{aligned}
\langle V_{\alpha_1}(0)V_{\alpha_2}(z)&V_{\alpha_3}(1)V_{\alpha_4}(\infty)\rangle^{\mathrm{cb}}=\frac{1}{8\pi}|z|^{2(\frac{Q^2}{4}-\Delta_1-\Delta_2)}\\&\times\int_\R|z|^{\frac{P^2}{2}}C_\gamma(\alpha_1,\alpha_2,Q-iP)C_\gamma(Q+iP,\alpha_3,\alpha_4)|\mc{F}_{\gamma,P}^{\alpha_{1234}}(z)|^2dP
\end{aligned}
\end{equation}
where $\mc{F}_{\gamma,P}^{\alpha_{1234}}(\cdot)$ is the four-point conformal block coming from the contribution of the OPE conformal block. It is also holomorphic in $z$ and universal in the sense that it depends only on $\gamma,P,\alpha_1,\alpha_2,\alpha_3$ and $\alpha_4$. We call this formula the \emph{conformal bootstrap equation}. The term ``bootstrap" refers to the fact that one can recursively compute all the correlation functions on any Riemann surface of any genus by ``bootstrapping" the structure constants using the spectrum and the conformal blocks.

Let us stress again that formula \eqref{eq:bootstrap} is far from having a mathematical justification. In general, one way to make sense of the bootstrap predictions is to recover them from the rigorous probabilistic framework of DKRV. This is usually a hard matter, but first steps have been made in this direction, notably in \cite{KRV1,KRV2} where the authors showed the validity of Ward identities and BPZ differential equations and gave a proof of the DOZZ formula. At this stage, we are still far from having a probabilistic interpretation of formula \eqref{eq:bootstrap} because the spectrum and the conformal blocks are not properly understood in the path integral approach. However, we will see that in the limit where $z\to0$, these two objects disappear from the equation and we are left with DOZZ formula which is well understood. 

\begin{figure}[h!]
\centering
\includegraphics[scale=0.7]{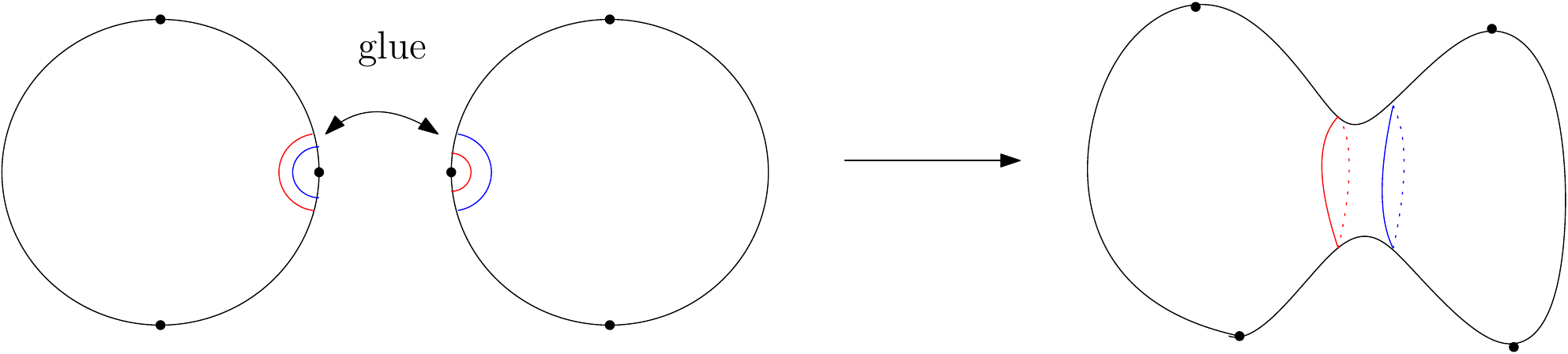}
\caption{\label{fig:glue}The gluing of two instances of the thrice-punctured sphere, producing a four-punctured sphere.}
\end{figure}

There is a geometric interpretation of equation \eqref{eq:bootstrap}. Indeed, one can produce a four-punctured sphere by gluing together two instances of the thrice-punctured sphere along annuli neighbourhoods of one puncture (see Figure \ref{fig:glue} and \cite{TV} for details of this procedure). The bootstrap equation is the CFT counterpart of this gluing procedure since the integrand is a product of DOZZ formulae. We will see in Section \ref{subsec:result} that the factorisation becomes exact in the $z\to0$ limit. The problem of factorisation of surfaces is an old one and was stressed by Seiberg (\cite{Sei} p336) as the most important open problem in Liouville CFT, at a time where the DOZZ formula was not yet known (nor even guessed). This paper gives a partial answer to the problem since we will show rigorously that the state factorises into two independent states as $z\to0$.
%

Finally, let us briefly comment on the place of this work within the existing literature. The recent proof of the DOZZ formula \cite{KRV2} made an extensive use of the BPZ equation, a second order ODE satisfied by the correlation function $z\mapsto\langle V_{-\frac{\gamma}{2}}(z)V_{\alpha_1}(0)V_{\alpha_2}(1)V_{\alpha_3}(\infty)\rangle$, which was established in the earlier paper \cite{KRV1} and solved explicitly using hypergeometric functions. The reason why such an equation was expected to hold in the first place is that the representation of the Virasoro algebra associated to the field $V_{-\frac{\gamma}{2}}(\cdot)$ is expected to be degenerate, with a null vector at level two in the Verma module. This drastically simplifies the fusion rule for the fields $V_{-\frac{\gamma}{2}}(z)V_{\alpha_1}(0)$, and using the interpretation of Virasoro generators as differential operators, this leads to the second order BPZ equation. In this paper on the contrary, we study the general form of the fusion rule, for which the associated representation should not be degenerate in general, thus not leading to a differential equation. To our knowledge, there is no rigorous construction of representations of the Virasoro algebra in Liouville CFT yet, but there are works addressing the question and exploiting null vectors in the context of boundary CFT. For instance, it was shown in \cite{dubedat} that SLE partitions functions can be constructed from highest-weight representations of the Virasoro algebra. In general, some BPZ and Ward-type identities appear in SLE related martingales as the condition making the drift term in It\^o's formula vanish \cite{friedrich}.  

	\subsection{Main results}
	\label{subsec:result}
Let $(\alpha_1,\alpha_2,\alpha_3,\alpha_4)$ be satisfying the Seiberg bounds \eqref{eq:seiberg_bounds}. In particular, this implies that either $\alpha_1+\alpha_2>Q$ or $\alpha_3+\alpha_4>Q$ (or both), and we assume without loss of generality that $\alpha_3+\alpha_4>Q$. Notice that these conditions are equivalent to having the Seiberg bounds being satisfied by $(\alpha_1,\alpha_2,Q)$ (with the exception of the $\alpha_3=Q$ saturation). 

Suppose for now that $\alpha_1+\alpha_2\geq Q$. Then equation \eqref{eq:bootstrap} is expected to hold, i.e. we should have

\begin{equation}
\begin{aligned}
\langle V_{\alpha_1}(z)V_{\alpha_2}(0)&V_{\alpha_3}(1)V_{\alpha_4}(\infty)\rangle^{\mathrm{cb}}=\frac{1}{8\pi}|z|^{2(\frac{Q^2}{4}-\Delta_1-\Delta_2)}\\
&\times\int_\R|z|^{\frac{P^2}{2}}C_\gamma(\alpha_1,\alpha_2,Q-iP)C_\gamma(Q+iP,\alpha_3,\alpha_4)|\mc{F}_{\gamma,P}^{\alpha_{1234}}(z)|^2dP
\end{aligned}
\end{equation}

At the geometrical level, we can produce a four-punctured sphere by gluing together two copies of the thrice-punctured sphere (see Figure \ref{fig:glue}) by picking one puncture on each sphere and identifying together annuli neighbourhoods of these punctures. The form of equation \eqref{eq:bootstrap} reveals this gluing construction: the four-point function is a factorisation of three-point functions.

Assume $\alpha_1+\alpha_2>Q$. Taking $\mc{F}_{\gamma,P}^{\alpha_{1234}}(z)\equiv1$ uniformly as $P\to0$, making the change of variable $P\mapsto P\sqrt{\log\frac{1}{|z|}}$, equation \eqref{eq:bootstrap} gives
\begin{equation}
\label{eq:bootstrap_limit}
\begin{aligned}
8\pi|z|^{2(\Delta_1+\Delta_2-\frac{Q}{4}^2)}&\langle V_{\alpha_1}(z)V_{\alpha_2}(0)V_{\alpha_3}(1)V_{\alpha_4}(\infty)\rangle^{\mathrm{cb}}\\
&=\int_{\R}|z|^{\frac{P^2}{2}}C_\gamma(\alpha_1,\alpha_2,Q-iP)C_\gamma(Q+iP,\alpha_3,\alpha_4)|\mc{F}_{\gamma,P}^{\alpha_{1234}}(z)|^2dP\\
&=\frac{1}{\sqrt{\log\frac{1}{|z|}}}\int_{\R}e^{-\frac{P^2}{2}}C_\gamma\left(\alpha_1,\alpha_2,Q-i\frac{P}{\sqrt{\log\frac{1}{|z|}}}\right)\\
&\qquad\qquad\qquad\times C_\gamma\left(Q+i\frac{P}{\sqrt{\log\frac{1}{|z|}}},\alpha_3,\alpha_4\right)\left|\mc{F}_{\gamma,\frac{P}{\sqrt{\log\frac{1}{|z|}}}}^{\alpha_{1234}}(z)\right|^2dP\\
&\underset{|z|\to0}{\sim}\left(\log\frac{1}{|z|}\right)^{-3/2}\partial_3C_\gamma(\alpha_1,\alpha_2,Q)\partial_1C_\gamma(Q,\alpha_3,\alpha_4)\int_\R P^2e^{-\frac{P^2}{2}}dP\\
&=\sqrt{2\pi}\left(\log\frac{1}{|z|}\right)^{-3/2}\partial_3C_\gamma(\alpha_1,\alpha_2,Q)\partial_1C_\gamma(Q,\alpha_3,\alpha_4)
\end{aligned}
\end{equation}
Hence
\begin{equation}
\label{eq:bootstrap_equiv}
\begin{aligned}
\langle V_{\alpha_1}(0)V_{\alpha_2}(z)V_{\alpha_3}(1)V_{\alpha_4}(\infty)\rangle^{\mathrm{cb}}\underset{z\to0}{\sim}\frac{|z|^{2(\frac{Q^2}{4}-\Delta_1-\Delta_2)}}{4\sqrt{2\pi}\log^{3/2}\frac{1}{|z|}}\partial_3C_\gamma(\alpha_1,\alpha_2,Q)\partial_1C_\gamma(Q,\alpha_3,\alpha_4)
\end{aligned}
\end{equation}

There are two important features in this asymptotic behaviour
\begin{itemize}
\item There is a $\left(\log\frac{1}{|z|}\right)^{-3/2}$ term correcting the polynomial rate $|z|^{2(\frac{Q^2}{4}-\Delta_1-\Delta_2)}$
\item The limit is expressed as a product of two derivative DOZZ formulae. Geometrically speaking, this means that we are sewing two instances of the thrice-punctured spheres, each one presenting a cusp at the $\alpha=Q$ singularity. The fact that we have a product means that we have two ``independent" surfaces.
\end{itemize}
\begin{figure}[h!]
\centering
\includegraphics[scale=0.6]{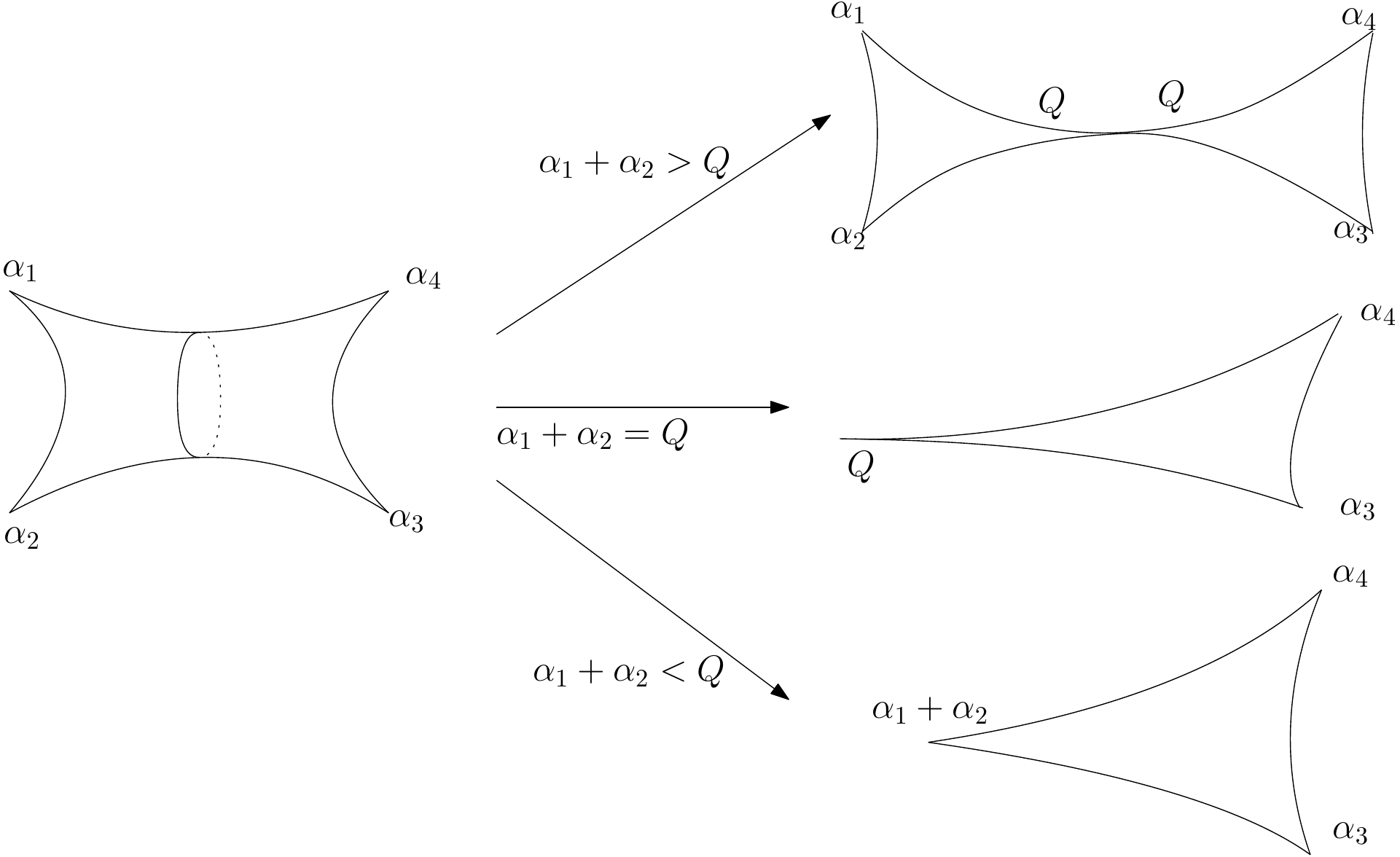}
\caption{\label{fig:three_regimes}The three different regimes depending on the sign of $\alpha_1+\alpha_2-Q$. \textbf{Up:} Case $\alpha_1+\alpha_2>Q$. The surface on the left is a four-punctured sphere with conical singularities of order $(\frac{\alpha_1}{Q},\frac{\alpha_2}{Q},\frac{\alpha_3}{Q},\frac{\alpha_4}{Q})$ at $(0,z,1,\infty)$. The limiting surface is a pair of thrice-punctured sphere: one with singularities $(\frac{\alpha_1}{Q},\frac{\alpha_2}{Q},1)$ at $(0,1,\infty)$ (the singularity at $\infty$ is a cusp), the other with singularities $(1,\frac{\alpha_3}{Q},\frac{\alpha_4}{Q})$ at $(0,1,\infty)$. \textbf{Middle:} Case $\alpha_1+\alpha_2=Q$. The limiting surface is a thrice-punctured sphere with singularities of order $(1,\frac{\alpha_3}{Q},\frac{\alpha_4}{Q})$ at $(0,1,\infty)$. \textbf{Bottom:} Case $\alpha_1+\alpha_2<Q$. The limiting surface is a thrice-punctured sphere with singularities $(\frac{\alpha_1+\alpha_2}{Q},\frac{\alpha_3}{Q},\frac{\alpha_4}{Q})$ at $(0,1,\infty)$.}
\end{figure}
In the case $\alpha_1+\alpha_2=Q$, the computation of Appendix \ref{app:dozz} shows that:
\begin{equation}
\label{eq:dozz_limit}
\underset{P\to0}{\lim}C_\gamma(\alpha_1,\alpha_2,Q-iP)C_\gamma(Q+iP,\alpha_3,\alpha_4)=-4\partial_1C_\gamma(Q,\alpha_3,\alpha_4)
\end{equation}

Going back to the bootstrap equation and noticing that $2(\frac{Q^2}{4}-\Delta_1-\Delta_2)=-\alpha_1\alpha_2$, we can apply the same change of variables as in \eqref{eq:bootstrap_limit}, and get in this case
\begin{equation}
\label{eq:bootstrap_limit_critical}
\begin{aligned}
\langle V_{\alpha_1}(z)V_{\alpha_2}(0)V_{\alpha_3}(1)V_{\alpha_4}(\infty)\rangle^{\mathrm{cb}}
&\underset{z\to0}{\sim}-\frac{|z|^{-\alpha_1\alpha_2}}{2\pi\sqrt{\log\frac{1}{|z|}}}\partial_1C_\gamma(Q,\alpha_3,\alpha_4)\int_\R e^{-\frac{P^2}{2}}dP\\
&=-\frac{1}{\sqrt{2\pi}}\frac{|z|^{-\alpha_1\alpha_2}}{\log^{1/2}\frac{1}{|z|}}\partial_1C_\gamma(Q,\alpha_3,\alpha_4)
\end{aligned}
\end{equation}
Again, let us notice two important features of this asymptotic behaviour
\begin{itemize}
\item There is a $\left(\log\frac{1}{|z|}\right)^{-1/2}$ correction term to be compared with the power $-3/2$ found in the supercritical case $\alpha_1+\alpha_2>Q$ in \eqref{eq:bootstrap_limit}. This is explained by the fact that there is only one cusp and one limiting surface (so no extra zero mode).
\item The limit is expressed with only one derivative DOZZ block, to be compared with the product found in \eqref{eq:bootstrap_limit}. Intuitively, this means that in this critical case $\alpha_1+\alpha_2=Q$, we see only one surface with two conical singularities and one cusp. 
\end{itemize}

Finally we turn to the case $\alpha_1+\alpha_2<Q$. In this case, equation \eqref{eq:bootstrap} does not hold in this form and there is a need for ``discrete corrections" (see \cite{BZ} section 8 for a thorough discussion of the phenomenon). This is linked with the fact that the contour of integration in \eqref{eq:bootstrap} includes poles of the DOZZ formula, and the discrete corrections are merely residues. 
In particular, the leading order as $z\to0$ is simply
\[\langle V_{\alpha_1}(0)V_{\alpha_2}(z)V_{\alpha_3}(1)V_{\alpha_4}(\infty)\rangle^\mathrm{cb}\underset{z\to\infty}{\sim}|z|^{-\alpha_1\alpha_2}C_\gamma(\alpha_1+\alpha_2,\alpha_3,\alpha_4)\]
so that the geometric interpretation is that the two singularities add up together. This makes sense since $(\alpha_1+\alpha_2,\alpha_3,\alpha_4)$ satisfies the Seiberg bounds. In this last case, the spectrum is ``hidden" behind the discrete leading-order terms. In order to see the spectrum in our probabilistic framework, one would need to push the asymptotic expansion further. It should be possible to do so using similar techniques as in \cite{KRV2} (section 6) but we restrict ourselves to the leading order for now.

\begin{theorem}
\label{thm:result}
Let $(\alpha_1,\alpha_2,\alpha_3,\alpha_4)$ satisfying the Seiberg bounds and such that $\alpha_3+\alpha_4>Q$. The asymptotic behaviour as $z\to0$ of the correlation function $\langle V_{\alpha_1}(0)V_{\alpha_2}(z)V_{\alpha_3}(1)V_{\alpha_4}(\infty)\rangle$ depends on the sign of $\alpha_1+\alpha_2-Q$ and is described by the following three cases.
\begin{enumerate}
\item \emph{Supercritical case:} 

If $\alpha_1+\alpha_2>Q$, then
\begin{equation}
\label{eq:larger}
\langle V_{\alpha_1}(0)V_{\alpha_2}(z)V_{\alpha_3}(1)V_{\alpha_4}(\infty)\rangle\underset{z\to0}{\sim}\frac{1}{4\sqrt{2\pi}}\frac{|z|^{2(\frac{Q^2}{4}-\Delta_1-\Delta_2)}}{\log^{3/2}\frac{1}{|z|}}\partial_3C_\gamma(\alpha_1,\alpha_2,Q)\partial_1C_\gamma(Q,\alpha_3,\alpha_4)
\end{equation}
\item \emph{Critical case:}

If $\alpha_1+\alpha_2=Q$, then
\begin{equation}
\label{eq:equal}
\langle V_{\alpha_1}(0)V_{\alpha_2}(z)V_{\alpha_3}(1)V_{\alpha_4}(\infty)\rangle\underset{z\to0}{\sim}-\frac{1}{\sqrt{2\pi}}\frac{|z|^{-\alpha_1\alpha_2}}{\log^{1/2}\frac{1}{|z|}}\partial_1C_\gamma(Q,\alpha_3,\alpha_4)
\end{equation}
\item \emph{Subcritical case\footnote{This was already proved in \cite{KRV2} section 6.1 and essentially follows from dominated convergence.}:}

If $\alpha_1+\alpha_2<Q$, then
\begin{equation}
\label{eq:lower}
\langle V_{\alpha_1}(0)V_{\alpha_2}(z)V_{\alpha_3}(1)V_{\alpha_4}(\infty)\rangle\underset{z\to0}{\sim}|z|^{-\alpha_1\alpha_2}C_\gamma(\alpha_1+\alpha_2,\alpha_3,\alpha_4)
\end{equation}
\end{enumerate}
\end{theorem}
The different regimes appearing in the statement of Theorem \ref{thm:result} have a natural geometric explanation (see Figure \ref{fig:three_regimes} for an illustration of the phenomenon). First, notice that the condition $\alpha_3+\alpha_4-Q>0$ corresponds to having the Seiberg bounds satisfied for $(Q,\alpha_3,\alpha_4)$, except that the first coefficient saturates the second bound. When $\alpha_1+\alpha_2<Q$, the two singularities add up and the limit is non-trivial. When $\alpha_1+\alpha_2=Q$, the second Seiberg bound is saturated and it is natural \cite{DKRV2,Bav} to expect the factor $(\log\frac{1}{|z|})^{-1/2}\partial_1C_\gamma(Q,\alpha_3,\alpha_4)$ since the $0^{\mathrm{th}}$ order is trivial in this case. When $\alpha_1+\alpha_2-Q>0$, this also explains the factor $(\log\frac{1}{|z|})^{-1}\partial_3C_\gamma(\alpha_1,\alpha_2,Q)\partial_1C_\gamma(Q,\alpha_3,\alpha_4)$. The extra $(\log\frac{1}{|z|})^{-1/2}$ term has a more subtle origin. Since both $(\alpha_1,\alpha_2,Q)$ and $(Q,\alpha_3,\alpha_4)$ satisfy the Seiberg bounds, we expect to see the two spheres split and form a disconnected pair of surfaces in the limit. In this limit, the GFF should have two zero modes (given e.g. by the mean on each independent surface). Roughly speaking, upon splitting, the mean on the right surface conditioned on the mean on the total surface is a Gaussian random variable with large variance which -- when properly rescaled -- produces the extra zero mode. This rescaling explains the extra $(\log\frac{1}{|z|})^{-1/2}$ term appearing in \eqref{eq:larger}.


Theorem \ref{thm:result} can be equivalently reformulated in terms of GMC. Since our proof does not depend on the particular choice of $(-\frac{Q\sigma}{\gamma})$-moment in the  four-point correlation, we may promote $\sigma$ to a free parameter and study fusion estimates for arbitrary negative moments of GMC that could be of independent interest. We first record the decay rate in the theorem below.


\begin{theorem} \label{theo:result_gmc}
Let $\kappa > 0$, $\gamma \in (0, 2)$ and $(\alpha_1, \alpha_2, \alpha_3, \alpha_4) \in \R_+^4$ be such that the Seiberg bound is satisfied. Also let $(z_1, z_2, z_3, z_4) = (0, z, 1, \infty)$ with $z \in \C \setminus \{0\}$. Then there exists some constant $E_{\kappa}^{\gamma}(\alpha_1, \alpha_2, \alpha_3, \alpha_4) > 0$ such that
\begin{align} \label{eq:gmc_const}
\lim_{z \to 0} \frac{1}{I_{\alpha_1 +\alpha_2}^{\gamma, \kappa}(z)} \E\left[ M^\gamma \left(e^{\gamma\sum_{j=1}^4 \alpha_j G(z_j, \cdot)}\right)^{-\kappa} \right] = E_{\kappa}^{\gamma}(\alpha_1, \alpha_2, \alpha_3, \alpha_4)
\end{align}

\noindent where the rate function $I_{\alpha}^{\gamma, \kappa}$ is given by
\begin{align*}
I_{\alpha}^{\gamma, \kappa} (z)
= \begin{cases}
1 & \alpha - Q < 0,\\
\sqrt{\log \frac{1}{|z|}} & \alpha - Q = 0, \\
|z|^{\frac{(\alpha- Q)^2}{2}} \left(\log \frac{1}{|z|}\right)^{3/2}  & \alpha -Q \in (0, \kappa \gamma) , \\
|z|^{\frac{(\alpha - Q)^2}{2}} \sqrt{\log \frac{1}{|z|}} & \alpha  -Q = \kappa \gamma , \\
|z|^{\frac{(\alpha - Q)^2}{2} - \frac{(\kappa \gamma - (\alpha - Q))^2}{2}} & \alpha  -Q > \kappa \gamma .
\end{cases}
\end{align*}

\end{theorem}

As mentioned in Section \ref{subsec:path_integral}, LCFT gives an exact expression for $E_{\kappa}^{\gamma}(\alpha_1, \alpha_2, \alpha_3, \alpha_4)$ in terms of the DOZZ formula when $\kappa = \frac{\sum_{i=1}^4 \alpha_i - 2Q}{\gamma}$. While this is not the case in general, we can still provide a probabilistic representation of the constant based on the radial/angular decomposition of the GFF on the infinite cylinder (see Section \ref{subsec:gff}). For this it is useful to introduce the random functional
\begin{equation}\label{eq:functional}
\begin{split}
F_{a_1, a_2}(u, f(\cdot))
& = e^{-\gamma u} \int_{|x| \ge 1} \frac{dM^\gamma(x)}{|x|^{4 - \gamma (a_1 + a_2)} |x-1|^{\gamma a_1}}
+\int_{\R_{s \ge 0} \times \S^1_\theta}e^{-\gamma(f(s) - a_1G(1, e^{-s - i\theta}))}d\widehat{M}^\gamma(s,\theta)\\
& = \int_{\mathcal{C}_\infty}e^{\gamma \left( (-u + B_s + (Q- a_2)s) 1_{\{s \le 0\}} - f(s)1_{\{s \ge 0\}} + a_1 G(1, e^{-s - i\theta}) \right)}d\widehat{M}^\gamma(s,\theta)
\end{split}
\end{equation}

\noindent where $(B_{-s})_{s \ge 0}$ is a Brownian motion independent of the GMC $d\widehat{M}^{\gamma}(s, \theta)$ associated with the lateral noise of GFF (see Lemma \ref{lem:radial_angular}). We will also write $(\widetilde{\beta}_{s}^u)_{s \ge 0}$ to denote a $\widetilde{\mathrm{BES}}_u(3)$-process (see Definition \ref{defi:con_Bessel}).

\begin{theorem}\label{theo:prob_rep1}
Let $\alpha_1 + \alpha_2 - Q \ge 0$. The constant $E_{\kappa}^{\gamma}(\alpha_1, \alpha_2, \alpha_3, \alpha_4)$ in \eqref{eq:gmc_const} has the following probabilistic representations.

\begin{itemize}[leftmargin=*]
\item If $\alpha_1 + \alpha_2 - Q = 0$, then
\begin{align}\label{eq:prob_rep2_1}
E_{\kappa}^{\gamma}(\alpha_1,  \alpha_2, \alpha_3, \alpha_4)
= \frac{1}{\kappa \gamma} \sqrt{\frac{2}{\pi}} \E\left[\left(F_{\alpha_3, \alpha_4}(\tau, \widetilde{\beta}_{\cdot}^{\tau})\right)^{-\kappa} \right]
\end{align}

\noindent where $\tau \sim \mathrm{Exp}(\kappa \gamma)$.

\item If $\alpha_1 + \alpha_2 - Q \in (0, \kappa \gamma)$, then 
\begin{align}
& E_{\kappa}^{\gamma}(\alpha_1, \alpha_2, \alpha_3, \alpha_4)
\notag = \frac{1}{\gamma}
\frac{B\left(\frac{\alpha_1 + \alpha_2 - Q}{\gamma}, \kappa - \frac{\alpha_1 + \alpha_2 - Q}{\gamma} \right) }{(\alpha_1 + \alpha_2 - Q)(\kappa \gamma - (\alpha_1 + \alpha_2 - Q))}
\sqrt{\frac{2}{\pi}}\\
\label{eq:prob_rep2_2}
& \qquad  \times
\E \left[\left(F_{\alpha_3, \alpha_4}(\tau, \widetilde{\beta}_{\cdot}^\tau)\right)^{-(\kappa - \frac{\alpha_1 + \alpha_2 - Q}{\gamma})}  \right]
\E \left[\left(F_{\alpha_2, \alpha_1}(\mathcal{T}, \widetilde{\beta}_{\cdot}^\mathcal{T})\right)^{-\frac{\alpha_1 + \alpha_2 - Q}{\gamma}}  \right]
\end{align}

\noindent where $\tau \sim \mathrm{Exp}(\kappa \gamma - (\alpha_1 + \alpha_2 - Q))$, $\mathcal{T} \sim \mathrm{Exp}(\alpha_1 + \alpha_2 - Q)$  and $B(x, y) = \frac{\Gamma(x)\Gamma(y)}{\Gamma(x+y)}$.

\item If $\alpha_1 + \alpha_2 - Q = \kappa \gamma$, then
\begin{align}\label{eq:prob_rep2_3}
E_{\kappa}^{\gamma}(\alpha_1, \alpha_2, \alpha_3, \alpha_4)
= \frac{1}{\kappa \gamma} \sqrt{\frac{2}{\pi}} \E\left[\left(F_{\alpha_2, \alpha_1}(\mathcal{T}, \widetilde{\beta}_{\cdot}^{\mathcal{T}})\right)^{-\kappa} \right].
\end{align}

\noindent where $\mathcal{T} \sim \mathrm{Exp}(\kappa \gamma)$.

\item If $\alpha_1 + \alpha_2 - Q > \kappa \gamma$, then
\begin{align}
\label{eq:prob_rep2_4}
E_{\kappa}^{\gamma}(\alpha_1, \alpha_2, \alpha_3, \alpha_4)
& = \E \left[ \left( F_{\alpha_2, \alpha_1}(0, -B_{\cdot}^{-(\alpha_1 + \alpha_2 - Q - \kappa\gamma)})\right)^{-\kappa}\right]
\end{align}

\noindent where $(B_s^{-(\alpha_1 + \alpha_2 - Q - \kappa \gamma)})_{s \ge 0}$ is a Brownian motion with negative drift $-(\alpha_1 + \alpha_2 - Q - \kappa \gamma)$.
\end{itemize}

\end{theorem}

\begin{remark}
When $\alpha_1 + \alpha_2 - Q > \kappa \gamma$, we can easily rewrite \eqref{eq:prob_rep2_4} as
\begin{align*}
E_{\kappa}^{\gamma}(\alpha_1, \alpha_2, \alpha_3, \alpha_4)
& = \E \left[ \left(\int_{\C}\frac{|x^{-1}|_+^{(\kappa+1)\gamma^2} dM^\gamma(x)}{|x|^{4 - \gamma (\alpha_1 + \alpha_2)}|x-1|^{\gamma \alpha_2}} \right)^{-\kappa}\right]
\end{align*}

\noindent which  is very similar to the subcritical regime $\alpha_1 + \alpha_2 - Q < 0$ where
\begin{align*}
E_{\kappa}^{\gamma}(\alpha_1, \alpha_2, \alpha_3, \alpha_4) = 
\E \left[ \left(\int_{\C} \frac{dM^\gamma(x)}{|x|^{\gamma(\alpha_1+\alpha_2)}|x-1|^{\gamma \alpha_3}|x|_+^{4 - \gamma  \sum_{j=1}^4 \alpha_j}}\right)^{-\kappa}\right]
\end{align*}

\noindent can be obtained immediately by dominated convergence.

\end{remark}


	\subsection{Conjectured link with random planar maps}
	\label{subsec:rpm}
The result of Theorem \ref{thm:result} has an interesting counterpart in the world of 2d discretised quantum gravity via the famous KPZ conjecture which was originally formulated in the physics literature by Knizhnik, Polyakov and Zamolodchikov \cite{KPZ}. Roughly speaking, the authors conjectured that, in some sense, LCFT should be the scaling limit of large random planar maps weighted by some statistical mechanics model. 

We start by recalling some facts about planar maps, using the setting of \cite{Ku} section 1 (see also \cite{DKRV1} section 5.3). A \emph{planar map} is a graph together with an embedding into the sphere such that no two edges cross and viewed up to orientation preserving homeomorphisms. 

For concreteness, we will work with triangulations, meaning that all the faces in the map are triangles. Let $\mc{T}_{N,3}$ be the set of planar triangulations with $N$ faces and 3 extra marked points (called \emph{roots}). The combinatorics of $\mc{T}_{N,3}$ is well known since the work of Tutte \cite{Tutte} and we have
\[\#\mc{T}_{N,3}\underset{N\to\infty}{\asymp}N^{-1/2}e^{-\mu_cN}\]
for some $\mu_c>0$. We mention that a wide class of planar maps fall into the same universality class (e.g. $2p$-angulations), meaning that they scale like $N^{-1/2}e^{-\mu_c N}$ where $\mu_c$ depends on the model.

 There is a way to conformally embed any triangulation $(\t,\x_1,\x_2,\x_3)$ into the sphere by first turning it into a topological manifold and second specifying complex coordinate charts. This endows the triangulation with a structure of Riemann surface with conical singularities at vertices with $n\neq6$ neighbours, and this embedding is unique if we add the extra requirement that the marked points $(\x_1,\x_2,\x_3)$ are sent to $(0,1,\infty)$ (see e.g. \cite{Ku}). Concretely, if $\bigtriangleup\subset\C$ is an equilateral triangle with unit (Lebesgue) volume, the embedding provides a conformal map $\psi_t:\bigtriangleup\to\widehat{\C}$ for each triangle $t$ in the map. For all $a>0$, we consider the pushforward measure $d\nu_{t,a}(z)=a^2|(\psi_t^{-1})'(z)|^2dz$ on $\psi_t(\bigtriangleup)$, which assigns a mass $a^2$ to each triangle of $\t$. The collection of $(\nu_{t,a})_{t\in\t}$ defines a measure $\nu^\t_a$ on $\widehat{\C}$, and in particular $\nu_a^\t(\widehat{\C})=Na^2$ for all $\t\in\mc{T}_{N,3}$.

The model becomes interesting when we choose the triangulation randomly. The simplest example is the case of \emph{pure gravity}, which amounts in sampling the triangulation with respect to the probability measure defined by
\[\P_a(\t,\x_1,\x_2,\x_3):=\frac{1}{Z_a}e^{-\mu|\t|}\]
where $\mu:=(1+a^2)\mu_c$, $|\t|$ is the number of faces of $\t$ and $Z_a$ is a normalising constant. Notice that $Z_a\to\infty$ as we send $a\to0$, which means that the measure selects larger and larger maps. When $(\t,\x_1,\x_2,\x_3)$ is sampled under $\P_a$, the KPZ conjecture states that the random measure $\nu_a=\nu_a^\t$ converges in distribution (with respect to the topology of weak convergence of measures) as $a\to0$ to a random Radon measure $\nu$ on $\S^2$. This limiting measure is expected to be given by the Liouville measure (see \cite{DKRV1} section 3.3 for a definition) and in particular, it should satisfy the property that for all measurable $A\subset\widehat{\C}$,
\[\E\left[\frac{\nu(A)}{\nu(\widehat{\C})}\right]=\int_Af_{\sqrt{8/3},\mu_c}\]
where we have defined the probability density function
\begin{equation}
\label{eq:prob_density}
f_{\gamma,\mu}(z):=\frac{\mu\gamma}{3\gamma-2Q}\frac{\langle V_\gamma(0)V_\gamma(z)V_\gamma(1)V_\gamma(\infty)\rangle}{C_\gamma(\gamma,\gamma,\gamma)}
\end{equation}
for all $\gamma\in(0,2)$ and $\mu>0$ (see Appendix \ref{app:constant} for the derivation of the normalising constant).
%
%
The critical case of Theorem \ref{thm:result} is given by $\gamma=\frac{2}{\sqrt{3}}$\footnote{We notice that this is a special value of $\gamma$ from the random maps perspective since it corresponds to the scaling limit of bipolar-oriented maps, see \cite{KMSW}}, so that $\gamma=\sqrt{\frac{8}{3}}$ falls into the supercritical case. Thus we have the asymptotic behaviour (note that $\Delta_\gamma=\frac{\gamma}{2}\times\frac{2}{\gamma}=1$)
\begin{equation}
f_{\gamma,\mu}(z)\underset{z\to0}{\sim}\frac{\mu\gamma}{2\sqrt{2\pi}(3\gamma-2Q)}|z|^{\frac{Q^2}{2}-4}\left(\log\frac{1}{|z|}\right)^{-3/2}\frac{(\partial_3C_\gamma(\gamma,\gamma,Q))^2}{C_\gamma(\gamma,\gamma,\gamma)}
\end{equation}

If we integrate this formula on a small disc of radius $\varepsilon$, we find
\[\int_0^\varepsilon r^{\frac{Q^2}{2}-4}\left(\log\frac{1}{r}\right)^{-3/2}rdr
=(Q^2/2-2)^{1/2}\int_{(Q^2/2-2)\log\frac{1}{\varepsilon}}^\infty e^{-u}u^{-3/2}du\underset{\varepsilon\to0}{\sim}2\frac{\varepsilon^{\frac{Q^2}{2}-2}}{\sqrt{\log\frac{1}{\varepsilon}}}\]
so that
\begin{equation}
\label{eq:density}
\int_{|z|\leq\varepsilon}f_{\gamma,\mu}(z)dz\underset{\varepsilon\to0}{\sim}\frac{\sqrt{2\pi}\mu\gamma}{3\gamma-2Q}\frac{(\partial_3C_\gamma(\gamma,\gamma,Q))^2}{C_\gamma(\gamma,\gamma,\gamma)}\frac{\varepsilon^{\frac{Q^2}{2}-2}}{\sqrt{\log\frac{1}{\varepsilon}}}
\end{equation}
If the conjecture holds true, the asymptotic behaviour \eqref{eq:density} gives the expected fraction of vertices which are close to 0 in a large planar map. In particular, the exponent of $\varepsilon$ is $\frac{Q^2}{2}-2=1/12$ for pure gravity.

Similar conjectures hold for random maps coupled with some statistical mechanics model (e.g. Ising, Potts... see \cite{DKRV1}). The conjectures are essentially the same in each case except that the value of $\gamma$ and $\mu$ may vary (e.g. Ising model corresponds to $\gamma=\sqrt{3}$). However one can still plug the good value of $\gamma$ in formula \eqref{eq:density} to conjecture the expected density of points around 0.

\subsection{Outline} The remainder of this article is organised as follows. In the next section, we provide a summary of GFF and GMC for the construction of Liouville correlation functions, and then explain the main idea of our proofs. Section \ref{sec:proof} is devoted to the proof of Theorem \ref{thm:result} (on the four-point correlation) and Theorem \ref{theo:result_gmc} (on the decay of arbitrary negative moments of GMC), while that of Theorem \ref{theo:prob_rep1} (on the probabilistic representations of the limiting constants) is treated in Section \ref{sec:proof_gmc}. In the appendices we collect the DOZZ formula, discuss our work from the perspective of surfaces with conical singularities and explain how to normalise the four-point correlation to a probability distribution.

\bigskip
\paragraph*{\textbf{Acknowledgements}} The authors wish to thank R\'emi Rhodes and Vincent Vargas for bringing this problem to their knowledge and for interesting comments on this work and discussions on LCFT.

\section{Background}
\label{sec:background}
In this section, we recall the mathematical foundation for the Liouville measure \eqref{eq:lcft_path_integral} and the derivation for the 4-point function, and explain the main idea of our approach. To commence with, we quickly review GFF and GMC and mention several facts about them.

	\subsection{Gaussian Free Field}
	\label{subsec:gff}
Let $H^1_0(\S^2,g)$ (or simply $H^1_0$) be the Sobolev space of functions with distributional derivatives in $L^2(\S^2,g)$ and vanishing $g$-mean. This space is equipped with the norm
\[\norm{X}_\nabla^2:=\frac{1}{2\pi}\int_{\S^2}|\nabla X|^2=-\frac{1}{2\pi}\int_{\S^2}\Delta X\cdot X\]
that we call the \emph{Dirichlet energy}. Hence we can interpret the formal measure $\frac{1}{Z_\mathrm{GFF}}\int e^{-\frac{1}{2}\norm{X}_\nabla^2}DX$ as a Gaussian probability measure on the space $H^1_0$ (where $Z_\mathrm{GFF}$ is a ``normalising constant" which we will explain at the end of this section). Thus if $(e_n)_{n\geq 1}$ is an orthonormal basis of $H^1_0$, we define the formal series
\[X=\sum_{n\geq 1}\alpha_ne_n\]
where $(\alpha_n)_{n\geq 1}$ is a sequence of i.i.d. normal random variables. It can be shown that this series converges in $H^{-1}_0$, the topological dual of $H^1_0$. In particular, it is not defined as a function but rather as a distribution in the sense of Schwartz. We call this field the Gaussian Free Field (GFF). We write $\P$ for the probability measure of the GFF and $\E$ the associated expectation. The covariance kernel of the GFF is given by Green's function $G:=(-\frac{1}{2\pi}\Delta)^{-1}$, i.e. we formally write 
\[\E[X(x)X(y)]=G(x,y)\]
where the kernel of Green's function is explicitly given by
\[G(x,y)=\log\frac{1}{|x-y|}+\log|x|_++\log|y|_+\]
Thus the ``normalising constant" $Z_\mathrm{GFF}$ that we are looking for should be given by $Z_\mathrm{GFF}:=(\det(-\frac{1}{2\pi}\Delta))^{1/2}$, which is obtained via zeta-regularisation \cite{OPS}.

There is a convenient choice of basis for $H^1_0$, which is the family $(\sqrt{\frac{2\pi}{\lambda_n}}\varphi_n)_{n\geq1}$ where $(\varphi_n)_{n\geq0}$ is an orthonormal basis of $L^2$ of eigenfunctions of $-\Delta$ with eigenvalues $0=\lambda_0<\lambda_1\leq...\leq\lambda_n...$. This gives an $L^2$ decomposition of the GFF, except that we are missing the zero mode (the coefficient in front of the constant function $\varphi_0\equiv \Vol_g(\S^2)^{-1/2}$). 	This should be a Gaussian with infinite variance and we interpret this as Lebesgue measure, since $\sqrt{\frac{2\pi}{\lambda}}$ times the law of a Gaussian random variable with variance $\lambda^{-1}$ converges vaguely to Lebesgue measure as $\lambda\to0$. So our final interpretation of the measure $e^{-\frac{1}{2}\norm{X}_\nabla^2}DX$ is that we set for all continuous functional $F:H^{-1}\to\R$
\begin{equation}
\int F(X)e^{-\frac{1}{2}\norm{X}_\nabla^2}DX=\left(\frac{\det(-\frac{1}{2\pi}\Delta)}{\Vol_g(\S^2)}\right)^{-1/2}\int_\R\E[F(X+c)]dc
\end{equation}

Throughout the paper, we will make an extensive use of the so-called \emph{radial/angular} decomposition of the GFF, which is better understood in cylinder coordinates. Let $\mc{C}_\infty:=\R_s\times\S^1_\theta$ be the complete cylinder. Under the conformal change of coordinates $\psi:z\mapsto-\log z$, the Riemann sphere $(\widehat{\C}\setminus\{0,\infty\},g)$ endowed with the cr\^epe metric is mapped to $(\cyl,g_\psi)$ with $g_{\psi}(s,\theta)=e^{-2|s|}$. From now on, we write $G$ for Green's function on $(\mc{C}_\infty,g_\psi)$ with vanishing mean on $\{0\}\times\S^1$.

\begin{lemma}
\label{lem:radial_angular}
Let $X(s,\theta)$ be a GFF on $\cyl$. Then we can write $X(s,\theta)=B_s+Y(s,\theta)$ where
\begin{enumerate}
\item $(B_s)_{s\in\R}$ is a two-sided Brownian motion. We will call this process the \emph{radial part} of the field.
\item $Y$ is a $\log$-correlated field with covariance kernel
\begin{equation}
\label{eq:def_H}
H(s,\theta,s',\theta'):=\E[Y(s,\theta)Y(s',\theta')]=\log\frac{e^{-s}\vee e^{-s'}}{|e^{-s-i\theta}-e^{-s'-i\theta'}|}
\end{equation}
We will call this field the \emph{lateral noise} or \emph{angular part} of the field. Notice that the law of $Y$ is translation invariant.
\item $B$ is independent of $Y$.
\end{enumerate}
\end{lemma}
Otherwise stated, Lemma \ref{lem:radial_angular} enables to rewrite Green's function (on the cylinder) as
\begin{equation}
\label{eq:decomposition_green}
\begin{aligned}
G(s,\theta,s',\theta')
&=(|s|\wedge|s'|)1_{ss'\geq0}+H(s,\theta,s',\theta')\\
&=(|s|\wedge|s'|)1_{ss'\geq0}+H(0,0,s'-s,\theta'-\theta)\\
&=(|s|\wedge|s'|)1_{ss'\geq0}+G(0,0,s'-s,\theta'-\theta)
\end{aligned}
\end{equation}
\begin{remark}
We will sometimes abuse notations and write the more compact form $G(s+i\theta,s'+i\theta')$ (resp. $H(s+i\theta,s'+i\theta')$) for $G(s,\theta,s',\theta')$ (resp. $H(s+i\theta,s'+i\theta')$).
\end{remark}

	\subsection{Gaussian Multiplicative Chaos}
	\label{subsec:gmc}
Recall that a GFF is only defined as a distribution, so the exponential term $e^{\gamma X}$ is ill-defined \textit{a priori}. However it is possible to make sense of the measure $e^{\gamma X(x)}g(x)d^2x$ using a regularising procedure based on Kahane's theory of Gaussian Multiplicative Chaos (GMC) (see \cite{RVb,RhVa,Be} for more detailed reviews). 

We use the regularisation called the circle average. For $\varepsilon>0$, let $X_{g,\varepsilon}$ be the average of $X$ on the geodesic circle of radius $\varepsilon$ in the metric $g$. The field $X_\varepsilon$ is continuous, so the measure
\[dM^\gamma_{g,\varepsilon}(x):=e^{\gamma X_{g,\varepsilon}(x)-\frac{1}{2}\gamma^2\E[X_{g,\varepsilon}(x)^2]}d^2x\]
is well defined for all $\gamma\in(0,2)$, and it is known that the sequence of measures $M^\gamma_{g,\varepsilon}$ converges weakly in probability to a (random) Radon measure $M^\gamma_g$ with no atoms.
%

An important property of GMC measure is its conformal covariance \cite{DKRV1,DRV,GRV} under conformal multiplication
\begin{proposition}
Let $\omega\in\mc{C}^\infty(\S^2,g)$. Let $X$ be a GFF on $(\S^2,g)$ and $M^\gamma_{\tilde{g}}$ be the GMC measure obtained when regularising the field with circle averages in the metric $\tilde{g}:=e^{\omega}g$. Then $M^\gamma_{\tilde{g}}=e^{\frac{\gamma Q}{2}}M^\gamma_g$.
\end{proposition}

\begin{remark}
For notational convenience, when the regularising metric is the background metric $g(x)=|x|_+^{-4}$ on $\widehat{\C}$, we will drop the subscript and write $M^\gamma=M^\gamma_g$.
\end{remark}

Another useful tool of GMC is Kahane's convexity inequality (\cite{RVb}, Theorem 2.2)
\begin{theorem}[Kahane 1985]
Let $X$ and $Y$ be two continuous Gaussian fields on $D\subset\S^2$ such that for all $x,y\in D$
\[\E[X(x)X(y)]\leq\E[Y(x)Y(y)]\]
Then for all convex function $F:\R_+\to\R$ with at most polynomial growth at infinity,
\[\E\left[F\left(\int_De^{\gamma X(x)-\frac{\gamma^2}{2}\E[X(x)^2]}d^2x\right)\right]\leq\E\left[F\left(\int_De^{\gamma Y(x)-\frac{\gamma^2}{2}\E[Y(x)^2]}d^2x\right)\right]\]
\end{theorem}
In practice, one can apply this theorem to the GMC measure $\log$-correlated fields like the GFF after using the regularising procedure.

Now suppose $X,Y$ are $\log$-correlated fields with $|\E[X(x)X(y)-\E[Y(x)Y(y)]|\leq\varepsilon$ and write $M^\gamma,N^\gamma$ for their respective chaos measure. In particular we have
\[\E[X(x)X(y)]\leq \E[Y(x)Y(y)]+\varepsilon\]
Notice that the field $Z(x)=Y(x)+\sqrt{\varepsilon}\delta$ -- with $\delta\sim\mc{N}(0,1)$ independent of everything -- has covariance kernel $\E[Y(x)Y(y)]+\varepsilon$. Hence by Kahane's convexity inequality, we have for all $\kappa>0$
\[\E[M^\gamma(D)^{-\kappa}]\leq \E[e^{-r\gamma\sqrt{\varepsilon}\delta}N^\gamma(D)^{-\kappa}]=e^{\frac{1}{2}\gamma^2r^2\varepsilon}\E[N^\gamma(D)^{-\kappa}]\]
By the symmetry of the roles played by $X$ and $Y$, the converse inequality is also true, so
\[\E[M^\gamma(D)^{-\kappa}]=\E[N^\gamma(D)^{-\kappa}](1+O_{\varepsilon\to0}(\varepsilon))\]
Similarly, we have for all $c\in\R$,
\[\E\left[\exp(-\mu	e^{\gamma c}M^\gamma(D))\right]=\E\left[\exp(-\mu e^{\gamma c}N^\gamma(D))\right](1+O_{\varepsilon\to0}(\varepsilon))\]
	
	\subsection{Derivation of the correlation function}
	\label{subsec:setup}
Using the GFF and GMC we are ready to state the definition of the correlation functions on the sphere. For $\varepsilon>0$, we can regularise the vertex operator $V_{\alpha_i}(z_i)$ by defining $V_{\alpha_i,\varepsilon}(z_i)=e^{\alpha_iX_\varepsilon(z_i)-\frac{\alpha_i^2}{2}\E[X_\varepsilon(z_i)^2]}$. By Cameron-Martin theorem, we have (recall $\sigma=\sum_{i=1}^N\frac{\alpha_i}{Q}-2>0$)
\begin{equation}
\label{eq:correl_regularised}
\left\langle\prod_{i=1}^NV_{\alpha,\varepsilon}(z_i)\right\rangle=2e^{C_\varepsilon(\z)}\int_\R e^{Q\sigma c}\E\left[\exp\left(-\mu e^{\gamma c}\int_{\widehat{\C}}e^{\gamma\sum_{i=1}^N\alpha_iG_\varepsilon(z_i,\cdot)}dM^\gamma\right)\right]dc
\end{equation}
where $C_\varepsilon(\z)=\sum_{i<j}\alpha_i\alpha_jG_\varepsilon(z_i,z_j)$. This regularised correlation function \eqref{eq:correl_regularised} converges to a positive finite limit as $\varepsilon\to0$ as long as the Seiberg bounds are satisfied as the GMC measure integrates the singularities around each insertion. We take this limit as our definition of the correlation function

\begin{equation}
\label{eq:def_correl}
\begin{aligned}
\left\langle\prod_{i=1}^NV_{\alpha_i}(z_i)\right\rangle
&=2e^{C(\z)}\int_\R e^{Q\sigma c}\E\left[\exp\left(-\mu e^{\gamma c}\int_{\widehat{\C}}e^{\gamma\sum_{i=1}^N\alpha_iG(z_i,\cdot)}dM^\gamma\right)\right]dc\\
&=2e^{C(\z)}\gamma^{-1}\mu^{-\frac{Q\sigma}{\gamma}}\Gamma\left(\frac{Q\sigma}{\gamma}\right)\E\left[\left(\int_{\widehat{\C}}e^{\gamma\sum_{i=1}^N\alpha_iG(z_i,\cdot)}dM^\gamma\right)^{-\frac{Q\sigma}{\gamma}}\right]
\end{aligned}
\end{equation}
after making the change of variable $u=e^{\gamma c}$. As can be seen from expression \eqref{eq:def_correl}, the finiteness of the correlation function in our probabilistic formulation is equivalent to the finiteness of the moments of the GMC measure. This holds provided the \emph{extended Seiberg bounds} are satisfied \cite{KRV2}
\[-\frac{Q\sigma}{\gamma}<\frac{4}{\gamma^2}\wedge\underset{1\leq i\leq N}{\min}(Q-\alpha_i)\qquad\qquad\forall i,\alpha_i<Q\]

 In particular, if $N=3$ with insertions at $(0,1,\infty)$ and Liouville momenta $(\alpha_1,\alpha_2,\alpha_3)$ satisfying the Seiberg bounds, the expression is simply
\begin{equation}
\begin{aligned}
\langle V_{\alpha_1}(0)V_{\alpha_2}(1)V_{\alpha_3}(\infty)\rangle
&=2\gamma^{-1}\mu^{-\frac{Q\sigma}{\gamma}}\Gamma\left(\frac{Q\sigma}{\gamma}\right)\E\left[\left(\int_{\widehat{\C}}e^{\gamma(\alpha_1G(0,\cdot)+\alpha_2G(1,\cdot)+\alpha_3G(\infty,\cdot))}dM^\gamma\right)^{-\frac{Q\sigma}{\gamma}}\right]
\end{aligned}
\end{equation}
and this expression equals the DOZZ formula $C_\gamma(\alpha_1,\alpha_2,\alpha_3)$ \cite{KRV2}.

As for the four-point correlation function with insertions at $(z_1,z_2,z_3,z_4)=(0,z,1,\infty)$ with $|z|<1$, we find
\begin{equation}
\begin{aligned}
\langle V_{\alpha_1}(0)V_{\alpha_2}(z)&V_{\alpha_3}(1)V_{\alpha_4}(\infty)\rangle\\
&=\frac{2}{|z|^{\alpha_1\alpha_2}|1-z|^{\alpha_2\alpha_3}}\int_\R e^{-Q\sigma c}\E\left[\exp\left(-\mu e^{\gamma c}\int_{\widehat{\C}}e^{\gamma\sum_{i=1}^4\alpha_iG(z_i,\cdot)}dM^\gamma\right)\right]dc
\end{aligned}
\end{equation}

	\subsection{Main idea}
	\label{subsec:main_idea}
We now explain our approach which is inspired by \cite{DKRV2}. By applying the radial/angular decomposition of the GFF as we will see in Section \ref{subsec:proof}, we can effectively transform our problem to the study of exponential functionals of Brownian motion. 

To be more precise consider the following toy model. Let $(B_s^{\lambda})_{s \ge 0}$ be a Brownian motion with drift $\lambda$, and suppose $C_1, C_2 > 0$ are two fixed constants. Our goal is to understand the asymptotics of
\begin{align}\label{eq:idea1}
\E \left[ \left(C_1 + \int_0^t e^{\gamma B_s^\lambda}ds + C_2 e^{\gamma B_t^\lambda}\right)^{-\kappa}\right]
\end{align}

\noindent as $t \to \infty$. In order to extract the leading order in \eqref{eq:idea1}, we have to play the game of balancing energy (i.e. asking our drifted Brownian motion $(B_s^\lambda)_s$ to remain small) and entropy (i.e. paying a multiplicative cost given by the probability of such event).
\begin{itemize}[leftmargin=*]
\item When $\lambda < 0$, we don't have to do anything because $B_s^\lambda \xrightarrow{s \to \infty} -\infty$ anyway, and 
\begin{align*}
\E \left[ \left(C_1 + \int_0^t e^{\gamma B_s^\lambda}ds + C_2 e^{\gamma B_t^\lambda}\right)^{-\kappa}\right] 
\xrightarrow{t \to \infty} \E \left[ \left(C_1 + \int_0^\infty e^{\gamma B_s^\lambda}ds \right)^{-\kappa}\right]
\end{align*} 

\noindent by dominated convergence easily.

\item When $\lambda = 0$, we should demand our Brownian motion to never exceed an $O(1)$ threshold. On the event that $\{\sup_{s \le t} B_s \le N \}$, $(N - B_s)_{s \le t}$ behaves like a $\mathrm{BES}_N(3)$-process and drifts to $-\infty$, and therefore for suitably chosen $t' \ll t$ we see that
\begin{align*}
C_1 + \int_0^t e^{\gamma B_s^\lambda}ds + C_2 e^{\gamma B_t^\lambda}
\approx C_1 + \int_0^{t'} e^{\gamma B_s^\lambda}ds 
\end{align*}

\noindent  is expected to be $O(1)$ while the entropy cost is given by
\begin{align*}
\P\left(\sup_{s \le t} B_s \le N\right) \sim \sqrt{\frac{2}{\pi}}\frac{N}{\sqrt{t}} =  O\left(t^{-\frac{1}{2}}\right).
\end{align*}

\item When $\lambda \in (0, \kappa \gamma)$, we still demand our drifted Brownian motion $B_t^{\lambda}$ to remain below an $O(1)$ threshold, which requires an entropy cost of
\begin{align*}
\P\left(\sup_{s \le t} B_s^{\lambda} \le N\right) 
\sim \sqrt{\frac{2}{\pi}} \frac{e^{-\frac{\lambda^2}{2}t}}{\lambda^2 t^{\frac{3}{2}}} N e^{\lambda N}
= O\left(e^{-\frac{\lambda^2}{2}t} t^{-\frac{3}{2}}\right).
\end{align*}

\noindent The structural difference here is that even though $B_s^{\lambda}$ is rather negative in the intermediate time interval $s \in [t', t-t']$, the terminal value $B_t^{\lambda}$ is typically $O(1)$:
\begin{align*}
\P\left( B_t^{\lambda} \le x \bigg| \sup_{s \le t} B_s^{\lambda} \le N \right)
\xrightarrow{t \to \infty} e^{-\lambda(N-x)}(1+ \lambda(N-x)), \qquad x \le N.
\end{align*}

\noindent Therefore for the purpose of deriving the renormalised constant, we will have to keep
\begin{align*}
C_1 + \int_0^{t} e^{\gamma B_s^\lambda}ds + C_2 e^{\gamma B_t^\lambda}
\approx C_1 + \int_0^{t'} e^{\gamma B_s^\lambda}ds + \int_{t-t'}^{t} e^{\gamma B_s^\lambda}ds + e^{\gamma B_t^{\lambda}} C_2.
\end{align*}

\noindent which is $O(1)$ as $(B_s^{\lambda})_{s \le t'}$ and $(B_{t-s}^{\lambda} - B_t^{\lambda})_{s \le t'}$ behave like the negation of two independent $\mathrm{BES}(3)$-processes.

\item Moving beyond, we can only ask the $B_s^{\lambda}$ not to drift faster than $\lambda - \kappa\gamma$ or else the entropy cost would be too expensive. To proceed we first apply Cameron-Martin theorem to rewrite \eqref{eq:idea1} as
\begin{align}
\notag
& \E \left[e^{\kappa \gamma B_t - \frac{\kappa^2 \gamma^2}{2}t} \left(C_1 + \int_0^t e^{\gamma B_s^{\lambda - \kappa\gamma}}ds + C_2 e^{\gamma B_t^{\lambda - \kappa\gamma}}\right)^{-\kappa}\right]\\
\label{eq:idea2}
& \qquad = e^{-\kappa \gamma \lambda t + \frac{\kappa^2\gamma^2}{2}t}
\E \left[\left(C_1e^{-\gamma B_t^{\lambda-\kappa\gamma}} + \int_0^t e^{\gamma (B_s^{\lambda - \kappa\gamma} - B_t^{\lambda - \kappa\gamma})}ds + C_2\right)^{-\kappa}\right].
\end{align}

 If $\lambda = \kappa\gamma$, there isn't any drift in the expectation. The observation from the case $\lambda = 0$ suggests that we may want to demand $B_{t-s} - B_t$ to not exceed an $O(1)$ threshold for $s \le t$. This would imply again an entropy cost of $O(t^{-\frac{1}{2}})$, and we expect that 
\begin{align*}
C_1e^{-\gamma B_t^{\lambda-\kappa\gamma}} + \int_0^t e^{\gamma (B_s - B_t)}ds + C_2
\approx \int_0^{t'} e^{\gamma (B_s - B_t)}ds + C_2
\end{align*}

\noindent is $O(1)$ because $(B_{t-s} - B_t)_{s \le t'}$ behaves like the negation of a $\mathrm{BES}(3)$-process as before.

If $\lambda > \kappa \gamma$, the story is simpler because $B_{t-s}^{\lambda - \kappa\gamma} - B_t^{\lambda - \kappa\gamma}$ may be seen as a Brownian motion with negative drift.  Similar to the earlier case where $\lambda < 0$,
\begin{align*}
C_1e^{-\gamma B_t^{\lambda-\kappa\gamma}} + \int_0^t e^{\gamma (B_s^{\lambda - \kappa\gamma} - B_t^{\lambda -\kappa \gamma})}ds + C_2
\approx \int_0^{t'} e^{\gamma (B_s^{\lambda - \kappa\gamma} - B_t^{\lambda -\kappa \gamma})}ds + C_2.
\end{align*}

\noindent is already $O(1)$ without incurring any further entropy cost.
\end{itemize}


	\subsection{Path decomposition of BES(3)-processes}
	\label{subsec:path_dec}
Before we proceed to the proofs, we collect Williams' path decomposition theorem \cite{Wil} for 3-dimensional Bessel processes (abbreviated as $\mathrm{BES}(3)$-processes) which will be helpful when we study the probabilistic representations of the renormalised constant \eqref{eq:gmc_const}.

\begin{theorem}[Williams 1974] \label{theo:path_dec}
Fix $x > 0$, and consider the following independent objects:
\begin{itemize}
\item $(B_s)_{s \ge 0}$ is a standard Brownian motion (starting from $0$).
\item $U$ is a $\mathrm{Uniform}[0,1]$ random variable.
\item $(\beta_s^0)_{s \ge 0}$ is a 3-dimensional Bessel process starting from $0$.
\end{itemize}

\noindent Then the process $(\widehat{\beta}_s^x)_{s\ge 0}$ defined by
\begin{align}\label{eq:path_dec}
\widehat{\beta}_s^x = \begin{cases}
x + B_s & s \le T_{-x(1-U)}, \\
xU + \beta_{s - T_{-x(1-U)}}^0 & s \ge T_{-x(1-U)}
\end{cases}
\end{align}

\noindent with
\begin{align*}
T_{-x(1-U)} 
= \inf \{s > 0: B_s = -x(1-U)\}
= \inf \{s > 0: x+ B_s = xU\}
\end{align*}

\noindent is a 3-dimensional Bessel process starting from $x$ (written as $\mathrm{BES}_x(3)$-process).
\end{theorem}

In view of Theorem \ref{theo:path_dec}, we introduce the following definition.
\begin{definition}\label{defi:con_Bessel}
Let $(B_s)_{s \ge 0}$ and $(\beta_s^0)_{s \ge 0}$ be as in Theorem \ref{theo:path_dec}, and $x \ge 0$ an independent random variable. Then the process $(\widetilde{\beta}_s^x)_{s \ge 0}$ defined by
\begin{align}\label{eq:con_Bessel}
\widetilde{\beta}_s^x = \begin{cases}
x + B_s & s \le T_{-x}, \\
\beta_{s - T_{-x}}^0 & s \ge T_{-x}
\end{cases}
\end{align}

\noindent with
\begin{align*}
T_{-x} 
= \inf \{s > 0: x+ B_s = 0\}
\end{align*}

\noindent is called a 3-dimensional Bessel process starting from $x$ conditioned on hitting $0$, written as $\widetilde{\mathrm{BES}}_x(3)$-process.
\end{definition}

\section{Proof of Theorem \ref{thm:result}}
\label{sec:proof}

	\subsection{Supercritical case}
	\label{subsec:proof}
 We set the insertions at $(z_1,z_2,z_3,z_4):=(0,z,1,\infty)$ with Liouville momenta $(\alpha_1,\alpha_2,\alpha_3,\alpha_4)$ satisfying the Seiberg bounds, and we write $-\log z=t+i\phi$ with $t>0$ and $\phi\in[0,2\pi)$. We assume that both $\alpha_3+\alpha_4-Q>0$ and $\alpha_1+\alpha_2-Q>0$ which corresponds to the supercritical case of Theorem \ref{thm:result}. Notice that this corresponds precisely to having $(\alpha_1,\alpha_2,Q)$ and $(Q,\alpha_3,\alpha_4)$ satisfying the Seiberg bounds (with respectively the 3$^{\mathrm{rd}}$ and $1^{\mathrm{st}}$ momenta saturating the second Seiberg bound).

\begin{proof}[Proof of \eqref{eq:larger}]
Let $X(s,\theta)=B_s+Y(s,\theta)$ be a GFF on $\cyl=\R_s\times\S^1_\theta$. By the conformal covariance of GMC, it is equivalent to study the chaos measure of $X$ with respect to $g_\psi$ or to consider the field $X(s,\theta)+\frac{Q}{2}\log g_\psi(s,\theta)=X(s,\theta)-Q|s|$ and do the regularisation with respect to Lebesgue measure. 

From now on, we write $d\widehat{M}^\gamma(s,\theta)$ for GMC measure of the lateral noise with respect to Lebesgue measure on $\cyl$ (while $dM^\gamma(x)$ will be used for GMC measure of the entire GFF in spherical coordinates).

We are interested in the total GMC mass
\begin{equation} \label{eq:W_t}
\begin{aligned}
W_t&:=\int_{\cyl}e^{\gamma(B_s+(\alpha_1-Q)s1_{s>0}-(\alpha_4-Q)s1_{s<0}+\alpha_3G(0,s+i\theta)+\alpha_2G(t+i\phi,s+i\theta))}d\widehat{M}^\gamma(s,\theta)\\
&=\int_{\cyl}e^{\gamma(B_s+(\alpha_1+\alpha_21_{s<t}-Q)s1_{s>0}-(\alpha_4-Q)s1_{s<0}+\alpha_3G(0,s+i\theta)+\alpha_2G(0,s-t+i(\theta-\phi)))}d\widehat{M}^\gamma(s,\theta)
\end{aligned}
\end{equation}

The behaviour of this integral is essentially governed by the radial process. From the expression above, we can see that on the negative real line the process is $(B_{-s}+(\alpha_4-Q)s)_{s\geq0}$ which is a Brownian motion with negative drift so the integrand is integrable at $s=-\infty$. On the positive real line, the radial process has a positive drift $\alpha_1+\alpha_2-Q$ up to time $t$, then a negative drift $\alpha_4-Q$ from $t$ to $\infty$. 

 The first step is to apply Cameron-Martin theorem to get rid of the $(\alpha_1+\alpha_2-Q)$ drift term in $[0,t]$, so that for all continuous and bounded function $F:\R\to\R$
\begin{equation}
\label{eq:rn}
\begin{aligned}
\E\left[F(W_t)\right]=\E\left[e^{(\alpha_1+\alpha_2-Q)B_t-\frac{1}{2}(\alpha_1+\alpha_2-Q)^2t}F(Z_t)\right]
\end{aligned}
\end{equation}
where $Z_t$ is the random variable defined by
\begin{equation}
Z_t:=\int_{\cyl}e^{\gamma(B_s+(\alpha_1-Q)(t-s)1_{s>t}-(\alpha_4-Q)s1_{s<0}+\alpha_2G(0,t-s+i(\phi-\theta))+\alpha_3G(0,s+i\theta))}d\widehat{M}^\gamma(s,\theta)
\end{equation}
Hence the correlation function takes the form (recall $t=\log\frac{1}{|z|}$)
\begin{equation}
\label{eq:correl}
\begin{aligned}
\langle V_{\alpha_1}(0)V_{\alpha_2}(z)&V_{\alpha_3}(1)V_{\alpha_4}(\infty)\rangle\\
&=2|z|^{2(\frac{Q^2}{4}-\Delta_1-\Delta_2)}|1-z|^{-\alpha_2\alpha_3}\int_\R e^{Q\sigma c}\E\left[e^{(\alpha_1+\alpha_2-Q)B_t}\exp(-\mu e^{\gamma c}Z_t)\right]dc
\end{aligned}
\end{equation}
where the exponent for $|z|$ was found by noticing that $\frac{1}{2}(\alpha_1+\alpha_2-Q)^2-\alpha_1\alpha_2=2(\frac{Q^2}{4}-\Delta_1-\Delta_2)$.

\begin{remark}
The change of measure \eqref{eq:rn} becomes trivial if $\alpha_1+\alpha_2=Q$. This is the reason why there is a phase transition at this value and why the case is easier to treat.
\end{remark}

\begin{remark}
From a geometric point of view, the change of measure \eqref{eq:rn} has the effect of changing the background metric from a cone to a cylinder as illustrated in Figure \ref{fig:rn} (see also Appendix \ref{app:conical} for links between changes of metrics and changes of probability measures).
\end{remark}

\begin{figure}[h!]
\centering
\includegraphics[scale=0.8]{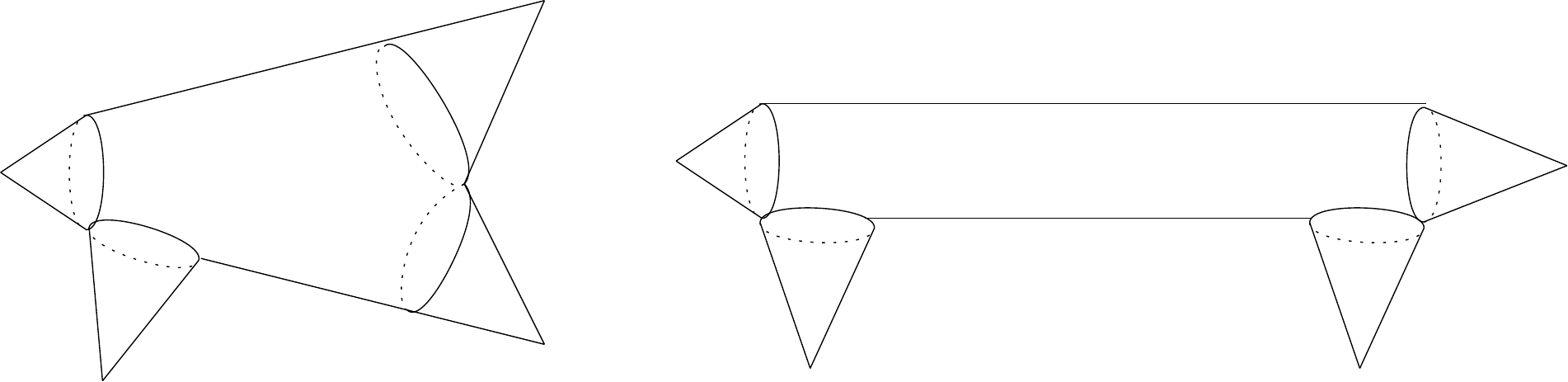}
\caption{\label{fig:rn}Change of measure from the cone to the cylinder}
\end{figure}

We can sample the radial part $(B_s)_{0\leq s\leq t}$ by the independent sum $B_s=Br_s+\frac{\delta}{\sqrt{t}}s$ where $(Br_s)_{0\leq s\leq t}$ is a standard Brownian bridge and $\delta\sim\mc{N}(0,1)$ (see Figure \ref{fig:field_decomposition}). We write $(\tilde{B}_s)_{0\leq s\leq t}$ the process on $\R$ where
\begin{enumerate}
\item $(\tilde{B}_{-s})_{s\geq 0}$ and $(\tilde{B}_s)_{s\geq t}$ are independent Brownian motions.
\item $(\tilde{B}_s)_{0\leq s\leq t}$ is a Brownian bridge in $[0,t]$ independent of the two other processes.
\end{enumerate}
 Similarly, we write $\tilde{Z}_t$ for the GMC mass defined similarly as $Z_t$ but with $\tilde{B}$ instead of $B$. The result will follow from an analysis of the behaviour of $\tilde{Z}_t$.
 
\begin{figure}[h!]
 \centering
 \includegraphics[scale=0.6]{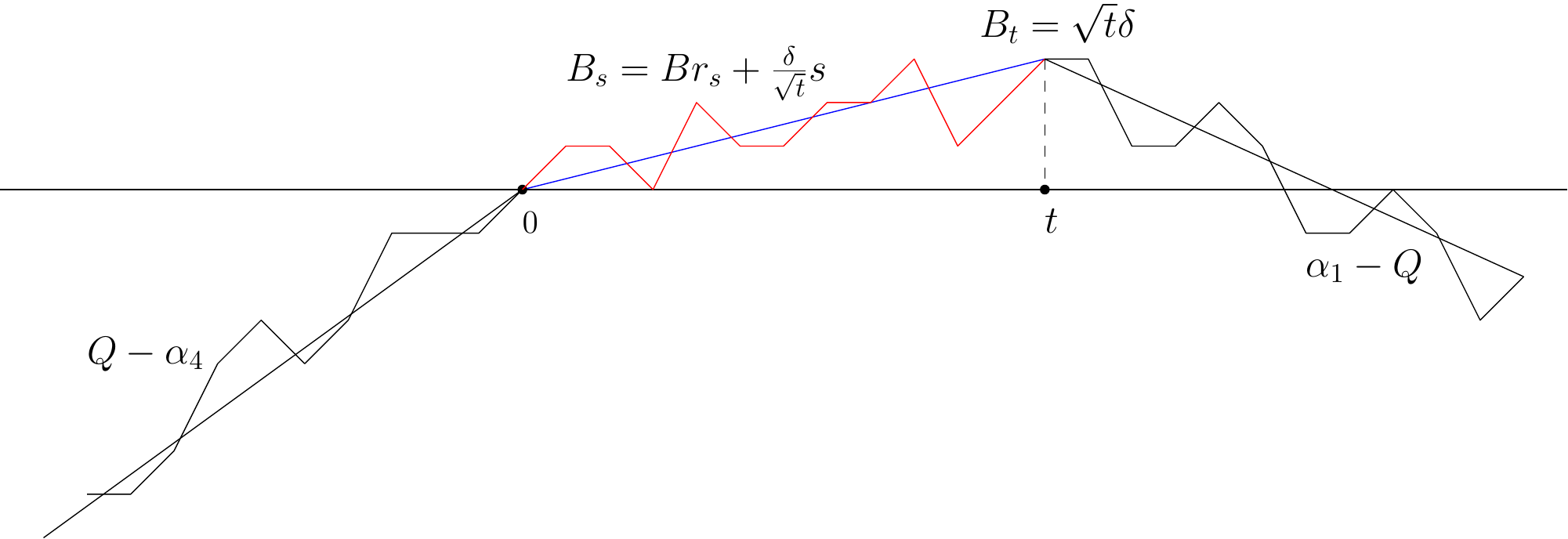}
 \caption{\label{fig:field_decomposition}The radial process in $(0,t)$ is the independent sum of a Brownian bridge (red) and a random drift (blue).}
 \end{figure}

 Let $\eta\in(0,1/2)$. We split $\tilde{Z}_t$ into three parts and write $\tilde{Z}_t=\tilde{L}_t+\tilde{C}_t+\tilde{R}_t$ where $\tilde{L}_t$, $\tilde{C}_t$ and $\tilde{R}_t$ are obtained by restricting the domain of integration to $(-\infty,t^{1/2-\eta})\times\S^1$, $(t^{1/2-\eta},t-t^{1/2-\eta})\times\S^1$ and $(t-t^{1/2-\eta},\infty)\times\S^1$ respectively. We define $Z_t=L_t+C_t+R_t$ similarly. These random variables are the ``left", ``central", and ``right" parts of the $\tilde{Z}_t$ and $Z_t$. 

For $b>0$, we introduce the event $\tilde{A}_{b,t}:=\left\lbrace\underset{0\leq s\leq t}{\sup}\tilde{B}_s \le b\right\rbrace$. This event has probability
\[\P(\tilde{A}_{b,t})=1-e^{-2b^2/t}=:f(b/\sqrt{t}).\]
Notice that $\underset{x\to\infty}{\lim}f(x)=1$ and $f(x)\underset{x\to0}{\sim}2x^2$.

Conditioning on $\tilde{A}_{b,t}$, the processes $(b-\tilde{B}_s)_{0\leq s\leq t/2}$ and $(b-\tilde{B}_{t-s})_{0\leq s\leq t/2}$ are absolutely continuous with respect to a BES$_b$(3)-process. Hence there exists $\eta'>0$ such that with high probability as $t\to\infty$, we have $\underset{t^{1/2-\eta}\leq s\leq t-t^{1/2-\eta}}{\sup}\tilde{B_s}\leq -t^{1/2-\eta'}$. It follows that $\tilde{C}_t\to0$ in probability as $t\to\infty$ when conditioned on $\tilde{A}_{b,t}$.

Let $\P_b$ the law of a field $X(s,\theta)=B_s+Y(s,\theta)$ where
\begin{enumerate}
\item $Y$ is a standard lateral noise.
\item $(B_{-s})_{s\geq0}$ is a standard Brownian motion.
\item $(b-B_s)_{s\geq0}$ is a BES$_b$(3)-process independent of $(B_{-s})_{s\geq0}$.
\end{enumerate}
We now describe the behaviour of $\tilde{L}_t$ and $\tilde{R}_t$.
 On $\tilde{A}_{b,t}$, the law of the process $(b-\tilde{B}_s)_{0\leq s\leq t^{1/2-\eta}}$ is absolutely continuous with respect to that of a BES$_b$(3)-process, and the Radon-Nikodym derivative tends to 1 a.s. and in $L^1$ as $t\to\infty$ (see e.g. \cite{MY} exercise 9.4). Hence the pair of processes $((b-\tilde{B}_s)_{0\leq s\leq t^{1/2-\eta}},(b-\tilde{B}_{t-s})_{0\leq s\leq t^{1/2-\eta}})$ converges in distribution to a pair of BES$_b$(3)-processes, and it is clear that these limit processes are independent of each other.
 
 As for the angular part, notice that for all $s<t^{1/2-\eta}$ and $s'>t-t^{1/2-\eta}$, we have for all $\theta,\theta'\in\S^1$,
 \begin{equation}
 \label{eq:green_estimate}
H(s+i\theta,s'+i\theta')=\log\frac{1}{|1-e^{-(s'-s)-i(\theta'-\theta)}|}\leq\log\frac{1}{1-e^{-(t-2t^{1/2-\eta})}}=O(e^{-t/2})
 \end{equation}
 
 Now let $Y^+,Y^-$ be independent lateral noises on $\cyl$ and define $Y'(s,\theta):=Y^+(s,\theta)1_{s<t/2}+Y^-(s,\theta)1_{s\geq t/2}$. Let $\tilde{L}_t^-$ (resp $\tilde{R}_t^+$) be the random variable defined like $\tilde{L}_t$ (resp. $\tilde{R}_t^-$) except we use $Y'$ rather than $Y$ for the lateral noise. Then under $\tilde{A}_{b,t}$, the pair $(\tilde{L}_t^-,\tilde{R}_t^+)$ converges in distribution to a pair of independent random variables $(L_\infty,R_\infty)$ with
\begin{equation*}
\begin{aligned}
&L_\infty\overset{\mathrm{law}}{=}\int_{\cyl}e^{\gamma(B_s-(\alpha_4-Q)s1_{s\leq0}+\alpha_3G(0,s+i\theta))}dM^\gamma(s,\theta)\\
&R_\infty\overset{\mathrm{law}}{=}\int_{\cyl}e^{\gamma(B_s-(\alpha_1-Q)s1_{s\leq0}+\alpha_2G(0,s+i\theta))}dM^\gamma(s,\theta)
\end{aligned}
\end{equation*} 
where the field is sampled from $\P_b$ in both cases. 
 
  Using the estimate \eqref{eq:green_estimate} and Kahane's convexity inequality, we have for all $c\in\R$
 \begin{equation*}
 \begin{aligned}
\E\left[\exp\left(-\mu e^{\gamma c}(\tilde{L}_t+\tilde{R}_t)\right)|\tilde{A}_{b,t}\right]
&=\E_b\left[\mc{E}_t\exp\left(-\mu e^{\gamma c}(\tilde{L}_t^-+\tilde{R}_t^+)\right)\right](1+O(e^{-t/2}))\\
&\underset{t\to\infty}{\to}\E_b\left[\exp(-\mu e^{\gamma c}(L_\infty+R_\infty))\right]\\
&=\E_b\left[\exp(-\mu e^{\gamma c}L_\infty)\right]\E_b\left[\exp(-\mu e^{\gamma c}R_\infty)\right]
\end{aligned}
 \end{equation*}

Putting pieces together, we find for all $c\in\R$
\begin{equation*}
\begin{aligned}
\underset{t\to\infty}{\lim}\E\left[\exp(-\mu e^{\gamma c}\tilde{Z}_t)|\tilde{A}_{b,t}\right]
&=\underset{t\to\infty}{\lim}\E\left[\exp(-\mu e^{\gamma c}\tilde{L}_t)\exp(-\mu e^{\gamma c}\tilde{C}_t)\exp(-\mu e^{\gamma c}\tilde{R}_t)|\tilde{A}_{b,t}\right]\\
&=\E_b\left[\exp(-\mu e^{\gamma c}L_\infty)\right]\E_b\left[\exp(-\mu e^{\gamma c}R_\infty)\right]
\end{aligned}
\end{equation*}

To conclude we need to relate the behaviour of $\tilde{Z}_t$ with that of $Z_t$ as $t\to\infty$. To this end we will condition on the value of the drift $\delta\sim\mc{N}(0,1)$. For fixed $\delta\in\R$, we have $\frac{\delta}{\sqrt{t}}t^{1/2-\eta}=\delta t^{-\eta}$, and this will be sufficient to show that up to time $t^{1/2-\eta}$, the radial part of the GFF $(B_{t-s}-\frac{\delta}{\sqrt{t}}s)_{0\leq s\leq t^{1/2-\eta}}$ does not ``feel" the drift and therefore looks like a Brownian motion started from $\sqrt{t}\delta$. More precisely, we have 
\[e^{-\gamma|\delta|t^{-\eta}}\tilde{R}_t\leq e^{-\gamma\sqrt{t}\delta}R_t\leq e^{\gamma|\delta|t^{-\eta}}\tilde{R}_t\]
Taking expectations and rescaling $\delta$ by $t^{-1/2}$, we get for all $c\in\R$
\begin{equation*}
\begin{aligned}
&\sqrt{t}\E\left[e^{(\alpha_1+\alpha_2-Q)B_t}\exp(-\mu e^{\gamma c}R_t)|\tilde{A}_{b,t}\right]\\
&\qquad\qquad=\int_\R e^{(\alpha_1+\alpha_2-Q)\delta}\E\left[\exp(-\mu e^{\gamma(c+\delta+\delta O(t^{-1/2-\eta}))}\tilde{R}_t)|\tilde{A}_{b,t}\right]\frac{e^{-\frac{t\delta^2}{2}}}{\sqrt{2\pi}}d\delta\\
&\qquad\qquad\underset{t\to\infty}{\to}\frac{1}{\sqrt{2\pi}}\int_\R e^{(\alpha_1+\alpha_2-Q)\delta}\E_b[\exp(-\mu e^{\gamma(c+\delta)}R_\infty]d\delta
\end{aligned}
\end{equation*}
where we applied the dominated convergence theorem in the last line. 
\begin{remark}
The take-out message of this computation is that as $t$ gets large the value of the radial part at $t$ is distributed like $\sqrt{t}\delta$, so when properly rescaled, its law converges vaguely to Lebesgue measure. Hence the field in the right part looks like a usual GFF plus a constant which is ``distributed" with Lebesgue measure, so $\delta$ plays the role of an extra zero mode in the limit. This translates the fact that we see two independent surfaces in the limit.
\end{remark}
Recalling the expression of the correlation function \eqref{eq:correl}, we make the change of variable $(c,\delta)=(u,v-u)$ (with Jacobian equal to 1) and find
\begin{equation}
\label{eq:lim_b}
\begin{aligned}
\sqrt{t}\int_\R &e^{Q\sigma c}\E\left[e^{(\alpha_1+\alpha_2-Q)B_t}\exp(-\mu e^{\gamma c}Z_t)|\tilde{A}_{b,t}\right]dc\\
&=\sqrt{t}\int_\R e^{Q\sigma c}\int_\R e^{(\alpha_1+\alpha_2-Q)\sqrt{t}\delta}\E\left[\exp(-\mu e^{\gamma c}Z_t)|\tilde{A}_{b,t}\right]\frac{e^{-\frac{\delta^2}{2}}}{\sqrt{2\pi}}d\delta dc\\
&\underset{t\to\infty}{\to}\frac{1}{\sqrt{2\pi}}\int_{\R^2}e^{(\alpha_1+\alpha_2-Q)(c+\delta)}e^{(\alpha_3+\alpha_4-Q)c}\E_b\left[\exp(-\mu e^{\gamma c}L_\infty)\right]\E_b\left[\exp(-\mu e^{\gamma (c+\delta)}R_\infty)\right]d\delta dc\\
&=\frac{1}{\sqrt{2\pi}}\left(\int_\R e^{(\alpha_3+\alpha_4-Q)u}\E_b\left[\exp(-\mu e^{\gamma u}L_\infty)\right]du\right)\left(\int_\R e^{(\alpha_1+\alpha_2-Q)v}\E_b\left[\exp(-\mu e^{\gamma v}R_\infty)\right]dv\right)
\end{aligned}
\end{equation}
Thus we have for each $b>0$
\begin{equation*}
\begin{aligned}
\underset{t\to\infty}{\lim}&t^{3/2}\int_\R e^{Q\sigma c}\E[\exp(-\mu e^{\gamma c}Z_t)|\tilde{A}_{b,t}]\P(\tilde{A}_{b,t})dc\\
&=\sqrt{\frac{2}{\pi}}b^2\left(\int_\R e^{(\alpha_1+\alpha_2-Q)u}\E_b[\exp(-\mu e^{\gamma u}R_\infty)]du\right)\left(\int_\R e^{(\alpha_3+\alpha_4-Q)v}\E_b[\exp(-\mu e^{\gamma v}L_\infty)]dv\right)
\end{aligned}
\end{equation*}

It is shown in \cite{DKRV2} that $b\E_b\left[\exp(-\mu e^{\gamma v}L_\infty)\right]$ has a non-trivial limit as $b\to\infty$ and, exchanging limits, the authors conclude that
\begin{equation}
\label{eq:some_limit}
\underset{b\to\infty}{\lim}b\E_b\left[\exp(-\mu e^{\gamma v}L_\infty)\right]=\underset{t\to\infty}{\lim}\sqrt{\frac{\pi t}{2}}\E\left[\exp(-\mu e^{\gamma v}L_t)\right]
\end{equation}
On the other hand, one can recover the BES$_b$(3)-process by conditioning a Brownian motion with negative drift to stay below $b$ forever and letting the drift tend to 0. More precisely, if $\tau_{\alpha,b}=\inf\{s\geq0,\;B_s+(\alpha-Q)s\geq b\}$, then we have $\P(\tau_{\alpha,b}=\infty)\underset{\alpha\to Q^-}{\sim}2(Q-\alpha)b$. Now adding the drift $\alpha-Q$ in the definition of $L_\infty$ gives the correlation function $\frac{1}{2}C_\gamma(\alpha,\alpha_3,\alpha_4)$. In the end (see \cite{Bav} for details), we have the alternative characterisation of the limit \eqref{eq:some_limit}
\begin{equation}
\label{eq:lim_is_dozz}
\underset{b\to\infty}{\lim}b\int_\R e^{(\alpha_3+\alpha_4-Q)v}\E_b\left[\exp(-\mu e^{\gamma v}L_\infty)\right]dv=-\frac{1}{4}\underset{\alpha\to Q}{\lim}\frac{C_\gamma(\alpha,\alpha_3,\alpha_4)}{\alpha-Q}=-\frac{1}{4}\partial_1C_\gamma(\alpha,\alpha_3,\alpha_4)
\end{equation}

A similar statement holds for the $L_\infty$ term, so we have
\begin{equation*}
\begin{aligned}
\underset{b\to\infty}{\lim}\underset{t\to\infty}{\lim}t^{3/2}\int_\R e^{Q\sigma c}\E[\exp(-\mu e^{\gamma c}Z_t)1_{\tilde{A}_{b,t}}]dc=\frac{1}{8\sqrt{2\pi}}\partial_3C_\gamma(\alpha_1,\alpha_2,Q)\partial_1C_\gamma(Q,\alpha_3,\alpha_4)
\end{aligned}
\end{equation*}

 From \cite{DKRV2}, the family of functions $\E[\exp(-\mu e^{\gamma c}Z_t)1_{\tilde{A}_{b,t}}]$ converges uniformly with respect to $t$ as $b\to\infty$, enabling us to exchange limits in $b$ an in $t$. Hence
\begin{equation*}
\begin{aligned}
\underset{t\to\infty}{\lim}t^{3/2}\int_\R e^{Q\sigma c}\E[\exp(-\mu e^{\gamma c}Z_t)]dc
&=\underset{b\to\infty}{\lim}\underset{t\to\infty}{\lim}t^{3/2}\int_\R e^{Q\sigma c}\E[\exp(-\mu e^{\gamma c}Z_t)1_{\tilde{A}_{b,t}}]dc\\
&=\frac{1}{8\sqrt{2\pi}}\partial_3C_\gamma(\alpha_1,\alpha_2,Q)\partial_1C_\gamma(Q,\alpha_3,\alpha_4)
\end{aligned}
\end{equation*}
Recall equation \eqref{eq:correl} to find
\begin{equation*}
\begin{aligned}
\langle V_{\alpha_1}(0)V_{\alpha_2}(z)&V_{\alpha_3}(1)V_{\alpha_4}(\infty)\rangle\\
&\underset{z\to0}{\sim}\frac{1}{4\sqrt{2\pi}}|z|^{2(\frac{Q^2}{4}-\Delta_1-\Delta_2)}|1-z|^{-\alpha_2\alpha_3}(\log\frac{1}{|z|})^{-3/2}\partial_3C_\gamma(\alpha_1,\alpha_2,Q)\partial_1C_\gamma(Q,\alpha_3,\alpha_4)
\end{aligned}
\end{equation*}
\end{proof}

	\subsection{Critical case}
	\label{subsec:proof_critical}
We conclude the proof of Theorem \ref{thm:result} by proving the asymptotic formula \eqref{eq:equal}, i.e. we assume $\alpha_1+\alpha_2=Q$. 
\begin{proof}[Proof of \eqref{eq:equal}]
The analysis of Section \ref{subsec:proof} fails only because the limit identified in \eqref{eq:lim_b} becomes trivial in this case because the triplet $(\alpha_1,\alpha_2,Q)$ violates the first Seiberg bound. Geometrically, the random variable $R_t$ does not have enough mass as $t\to\infty$ in order to produce another surface.

However, the analysis is still valid up to equation \eqref{eq:rn} and the expression of $Z_t$ is the same with this new set of parameters. Consider the same decomposition $Z_t=L_t+C_t+R_t$ and write $\xi_t:=C_t+R_t$ with the same $\eta>0$.

 As before, we condition the radial part no to exceed a given value. For $b>0$, we define the event
\[A_{b,t}:=\left\lbrace\underset{0\leq s\leq t}{\sup}B_s \le b\right\rbrace\]
It is well-known that 
\[\P(A_{b,t})=\sqrt{\frac{2}{\pi}}\int_0^{b/\sqrt{t}}e^{-\frac{x^2}{2}}dx=:g(b/\sqrt{t})\]
Notice that $g(x)\underset{x\to\infty}{\to}1$ and $g(x)\underset{x\to0}{\sim}\sqrt{\frac{2}{\pi}}x$. The process $(B_s)_{s\geq0}$ conditioned on $A_{b,t}$ has the law of a BES$_b$(3)-process. Repeating the argument of the previous subsection, we find that $\xi_t\to0$ in probability as $t\to\infty$ when conditioned on $A_{b,t}$.

As for the radial part, we have the following estimate for $s<t^{1/2-\eta}$ and $\theta\in\S^1$
\[|H(s+i\theta,t+i\phi)|=\log\frac{1}{|1-e^{-(t-s)-i(\phi-\theta)}|}=O(e^{-t/2})\]

Let $\P_b$ be the law of the field when the radial part $(B_s)_{s\geq0}$ is conditioned not to exceed $b$. Applying exactly the same framework as before, we have for all $\kappa>0$
\begin{equation}
\label{eq:lim_critical_dozz}
\begin{aligned}
\underset{t\to\infty}{\lim}\sqrt{t}\E\left[Z_t^{-\kappa}\right]
&=\underset{t\to\infty}{\lim}\sqrt{t}\E\left[L_t^{-\kappa}\right]\\
&=\sqrt{\frac{2}{\pi}}\underset{b\to\infty}{\lim}b\E_b\left[L_t^{-\kappa}\right]\\
\end{aligned}
\end{equation}
So it follows from the result of \cite{Bav} that
\begin{equation}
\underset{t\to\infty}{\lim}\int_\R e^{-Q\sigma c}\E\left[\exp\left(-\mu e^{\gamma c}\int_{\widehat{\C}}e^{\gamma\sum_{i=1}^4\alpha_iG(z_i,\cdot)}dM^\gamma\right)\right]dc=-\frac{1}{2\sqrt{2\pi}}\partial_1C_\gamma(Q,\alpha_3,\alpha_4)
\end{equation}
which concludes the proof.
\end{proof}

\subsection{Proof of Theorem \ref{theo:result_gmc}} \label{subsec:gmc_proof}
As mentioned in Section \ref{subsec:result}, Theorem \ref{theo:result_gmc} follows easily from Theorem \ref{thm:result} by taking $\sigma$ to be arbitrary. We will use the notations in Section \ref{subsec:proof} and \ref{subsec:proof_critical}, outlining the differences with the Liouville case and leaving the details to the reader.

Let $(\alpha_1,\alpha_2,\alpha_3,\alpha_4)$ be such that the Seiberg bound is satisfied. 
If $\alpha_1+\alpha_2-Q< \kappa \gamma$, the previous analysis applies immediately modulo the obvious substitution $\frac{Q\sigma}{\gamma} \leftrightarrow \kappa$ in the relevant places. If $\alpha_1+\alpha_2-Q \ge \kappa \gamma$, however, we only apply Cameron-Martin to partially offset the positive drift in $[0, t]$ by $\kappa \gamma$, as motivated in Section \ref{subsec:main_idea}. This leads to
\begin{align} \label{eq:rn2}
\E \left[ W_t^{-\kappa}\right]
= e^{-\kappa \gamma (\alpha_1 + \alpha_2 - Q)t + \frac{\kappa^2 \gamma^2}{2}t }\E \left[\left(e^{-\gamma \left(B_t + (\alpha_1 + \alpha_2 - Q - \kappa \gamma)t\right)}\widehat{Z}_t\right)^{-\kappa} \right]
\end{align}

\noindent where $W_t$ is defined in \eqref{eq:W_t} and $\widehat{Z}_t$ is defined suitably. Notice that \eqref{eq:rn2} is identical to \eqref{eq:rn} when $\alpha_1 + \alpha_2 - Q = \kappa \gamma$, the analysis of which is similar to that of Section \ref{subsec:proof_critical} except that here we consider the event
\begin{align*}
A_{b, t}' := \left\{ \sup_{0 \le s \le t} (B_{t-s} - B_t) \le b\right\}
\end{align*}

\noindent so that $L_t$ becomes irrelevant in the limit while $R_t$ survives as $t \to \infty$ instead. The case $\alpha_1 + \alpha_2 - Q > \kappa \gamma$ is straightforward because $e^{-\gamma \left(B_t + (\alpha_1 + \alpha_2 - Q - \kappa \gamma)t\right)}\widehat{Z}_t$ is an integral involving the exponentiation of a two-sided Brownian motion with negative drifts in both directions, and we can even obtain \eqref{eq:prob_rep2_4} by dominated convergence directly.


\qed

\section{Proof of Theorem \ref{theo:prob_rep1}} \label{sec:proof_gmc}
The rest of this paper is devoted to the proof of Theorem \ref{theo:prob_rep1} which gives probabilistic representations for the limits \eqref{eq:lim_is_dozz} and \eqref{eq:lim_critical_dozz} for which we do not have exact formulae outside of the Liouville case. We will not discuss \eqref{eq:prob_rep2_4} which is basically explained in the last section.
	
	\subsection{Infinite series representation of $E_{\kappa}^{\gamma}(\alpha_1, \alpha_2, \alpha_3, \alpha_4)$}
In order to obtain Theorem \ref{theo:prob_rep1} we need the following intermediate result.
\begin{lemma}\label{lem:prob_rep}
Fix $h > 0$. When $\alpha_1 + \alpha_2 - Q \in [0, \kappa \gamma]$, the constant $E_{\kappa}^{\gamma}(\alpha_1, \alpha_2, \alpha_3, \alpha_4)$ in \eqref{eq:gmc_const} has the following representations.
\begin{itemize}[leftmargin=*]
\item If $\alpha_1 + \alpha_2 - Q = 0$, we have
\begin{align}\label{eq:prob_rep1_1}
E_{\kappa}^{\gamma}(\alpha_1, \alpha_2, \alpha_3, \alpha_4)
= \sqrt{\frac{2}{\pi}}\sum_{n=1}^\infty nh e^{-\kappa \gamma nh} \E  \left[ \left(F_{\alpha_3, \alpha_4}(nh, \beta_{\cdot}^{nh})\right)^{-\kappa} 1_{\{\min_{s > 0} \beta_s^{nh} \le h\}}\right]
\end{align}

\noindent where $(\beta_s^u)_{s \ge 0}$ is a $\mathrm{BES}_u(3)$-process.

\item If $\alpha_1 + \alpha_2 - Q \in (0, \kappa \gamma)$, 
\begin{align}\label{eq:prob_rep1_2}
E_{\kappa}^{\gamma}(\alpha_1, \alpha_2, \alpha_3, \alpha_4)
=\sqrt{\frac{2}{\pi}} \sum_{n=1}^\infty \frac{ nh e^{-(\kappa \gamma - (\alpha_1 + \alpha_2 - Q))nh}}{(\alpha_1 + \alpha_2 - Q)^2 } 
\E \left[ \frac{1_{\{\min_{s > 0} \beta_{L, s}^{nh} \le h\}\cup \{\min_{s > 0} \beta_{R, s}^{\mathcal{T}} \le h\}}}{\left(F_{\alpha_3, \alpha_4}(nh, \beta_{L, \cdot}^{nh}) + F'_{\alpha_2, \alpha_1}(\mathcal{T}, \beta_{R, \cdot}^{\mathcal{T}})\right)^{\kappa}}
\right]
\end{align}

\noindent where $(\beta_{L, s}^u)_{s \ge 0}$ and $(\beta_{R, s}^\mathcal{T})_{s\ge0}$ are independent $\mathrm{BES}_u(3)$- and $\mathrm{BES}_{\mathcal{T}}(3)$-processes respectively with $\mathcal{T} \sim \mathrm{Gamma}(2, \alpha_1 + \alpha_2 - Q)$, and $F'$ is an independent copy of $F$.

\item If $\alpha_1 + \alpha_2 - Q = \kappa \gamma$, 
\begin{align}\label{eq:prob_rep1_3}
E_{\kappa}^{\gamma}(\alpha_1, \alpha_2, \alpha_3, \alpha_4)
= \sqrt{\frac{2}{\pi}}\sum_{n=1}^\infty nh e^{-\kappa \gamma nh} \E  \left[ \left(F_{\alpha_2, \alpha_1}(nh, \beta_{\cdot}^{nh})\right)^{-\kappa} 1_{\{\min_{s > 0} \beta_s^{nh} \le h\}}\right]
\end{align}
\noindent where $(\beta_s^u)_{s \ge 0}$ is a $\mathrm{BES}_u(3)$-process.

\end{itemize}
\end{lemma}

\begin{proof}
For the sake of brevity we only sketch the proof for the case $h = 1$ here and leave the details to the reader. The key idea is the partitioning of
\begin{align*}
A_{n, t}
= \left\{\sup_{0 \le s \le t} B_s \le n \right\}
= \bigcup_{k \le n} \left\{\sup_{0 \le s \le t} B_s \in [(k-1), k] \right\}
= \bigcup_{k \le n} \left\{\min_{0 \le s \le t} k - B_s \in [0, 1] \right\}.
\end{align*}

When $\alpha_1 + \alpha_2 - Q = 0$, our claim essentially follows from Proposition 3.1 and Lemma 3.2 in \cite{DKRV2}, where a dominated convergence argument (see the paragraph after Lemma 3.2 and Section 5.0.3 in that article) implies that the renormalised constant is given by
\begin{align*}
\sum_{n=1}^\infty \lim_{t \to \infty} \left(\sqrt{t} \E \left[L_t^{-\kappa} 1_{\{\min_{0 \le s \le t} n - B_s \le 1 \}} \big| A_{n, t}\right] \P(A_{n, t})\right)
= \sqrt{\frac{2}{\pi}}\sum_{n=1}^\infty  n \E_n \left[L_{\infty}^{-\kappa} 1_{\{\min_{s \ge 0} n - B_s \le 1\}}\right]
\end{align*}

\noindent which is equivalent to \eqref{eq:prob_rep1_1}. The proof of \eqref{eq:prob_rep1_3} is similar.

To apply the same dominated convergence approach to \eqref{eq:prob_rep1_2}, we need a control analogous to \cite[equation (3.18)]{DKRV2} when $\alpha_1 + \alpha_2 - Q \in (0, \kappa \gamma)$. Indeed the same argument there suggests that
\begin{align*}
t^{3/2} \E\left[ e^{(\alpha_1 + \alpha_2 - Q)B_t} (L_t + R_t)^{-\kappa} 1_{\{ \sup_{0 \le s \le t} B_s \in [(n-1), n]\}}\right]
\le C e^{-(\kappa \gamma - (\alpha_1 + \alpha_2 - Q))n}
\end{align*}

\noindent for some constant $C > 0$ independent of $t$ and $n$, and therefore $E_{\kappa}^{\gamma}(\alpha_1, \alpha_2, \alpha_3, \alpha_4)$ again has an infinite series representation of the form 
\begin{align}\label{eq:proof_prob_rep1}
\sum_{n=1}^\infty \lim_{t \to \infty} \left(t^{3/2} \E\left[ e^{(\alpha_1 + \alpha_2 - Q)B_t} (L_t + R_t)^{-\kappa} 1_{\{\sup_{0 \le s \le t} B_s \in [n-1, n] \}} \right]\right).
\end{align}

Let us highlight several observations. 
\begin{itemize}[leftmargin=*]
\item For every $n \in \N$, the event $\{\sup_{0 \le s \le t} B_s \in [n-1, n] \}$ may be replaced by 
\begin{align*}
\underbrace{\left\{\sup_{0 \le s \le t} B_s \le n\right\}}_{= A_{n, t} } \cap 
\underbrace{\left( \left\{\min_{0 \le s \le t^{1/2 - \eta}} n - B_s \le 1 \right\} \cup \left\{\min_{0 \le s \le t^{1/2-\eta}} n - B_t - (B_{t-s} - B_t) \le 1 \right\}\right)}_{=: \overline{A}_{n, t}}
\end{align*}

\noindent up to a cost of $o(1)$ for neglecting the unlikely event $\left\{\sup_{s \in [ t^{1/2-\eta}, t- t^{1/2-\eta}]} B_s \ge n-1\right\}$.

\item  Similar to the proof of Theorem \ref{thm:result}, if we condition on the event $A_{n, t}$ and $B_t = x$, then 
\begin{align*}
(n-B_s)_{0 \le s \le t^{1/2-\eta}},
\qquad (n - B_t - (B_{t-s} - B_t))_{0 \le s \le t^{1/2-\eta}}
\end{align*}
converge in distribution to independent $\mathrm{BES}_n(3)$- and $\mathrm{BES}_{n-x}(3)$-processes $(\beta_{L, s}^n)_{s \ge 0}$ and $(\beta_{R, s}^{n-x})_{s \ge 0}$ respectively. 
Consequently $L_t$ and $R_t$ converge in distribution to $e^{\gamma n}F_{\alpha_3, \alpha_4}(n, \beta_{L, \cdot}^n)$ and $e^{\gamma n}F'_{\alpha_2, \alpha_1}(n-x, \beta_{R, \cdot}^{n-x})$ respectively.

\end{itemize}

We now compute
\begin{align*}
& \E \left[1_{A_{n, t} \cap \overline{A}_{n, t}} \big| (B_s)_{s \in (-\infty, t^{1/2-\eta}] \cup [t-t^{1/2-\eta}, \infty)}\right] \\
& \qquad = 1_{ \left\{\min_{0 \le s \le t^{1/2-\eta}} n - B_s \le 1 \right\} \cup \left\{\min_{0 \le s \le t^{1/2-\eta}} n - B_t - (B_{t-s} - B_t) \le 1 \right\}}\\
& \qquad \qquad \times \P\left(A_{n, t} \big| (B_s)_{s \in (-\infty, t^{1/2-\eta}] \cup [t-t^{1/2-\eta}, \infty)} \right)
\end{align*}

\noindent where 
\begin{align*}
&\P\left(A_{n, t} \big| (B_s)_{s \in (-\infty, t^{1/2-\eta}] \cup [t-t^{1/2-\eta}, \infty)} \right)
= 1_{\{\sup_{0 \le s \le t^{1/2-\eta}} B_s \le n\}} 1_{\{\sup_{0 \le s \le t^{1/2-\eta}} B_{t-s} - B_t \le n - B_t\}}\\
& \qquad \qquad \qquad  \times \P\left(\sup_{t^{1/2-\eta} \le s \le  t - t^{1/2-\eta}} B_s \le n \bigg| B_{t^{1/2-\eta}}, B_{t-t^{1/2-\eta}}\right)
\end{align*}

\noindent and
\begin{align*}
\P\left(\sup_{t^{1/2-\eta} \le s \le  t - t^{1/2-\eta}} B_s \le n \bigg| B_{t^{1/2-\eta}}, B_{t-t^{1/2-\eta}}\right)
=  1- e^{-\frac{2}{t - 2t^{1/2 - \eta}}(n - B_{t^{1/2-\eta}})(n - B_t - (B_{t -t^{1/2-\eta}} - B_t))}
\end{align*}

\noindent is asymptotically $\frac{2}{t} (n - B_{t^{1/2-\eta}})(n - B_t - (B_{t -t^{1/2-\eta}} - B_t))$ when $t$ is large. In particular 
\begin{align*}
\P\left(A_{n, t} \big| B_t = x\right) \sim \frac{2}{t}n(n-x) + o(t^{-1}), \qquad t \to \infty.
\end{align*}

\noindent Substituting this into the summand in \eqref{eq:proof_prob_rep1}, we obtain
\begin{align*}
& \lim_{t \to \infty} t^{3/2} \int_{-\infty}^n 
\E \left[ e^{(\alpha_1 + \alpha_2 - Q)x} (L_t + R_t)^{-\kappa} 1_{\overline{A}_{n, t}} \bigg| A_{n, t}, B_t = x\right] \P(A_{n, t} \big| B_t = x) \P(B_t \in dx)\\
& = \frac{e^{(\alpha_1 + \alpha_2 - Q)n}}{\sqrt{2\pi}}\lim_{t \to \infty} t \int_{-\infty}^n 
\E \left[ e^{-(\alpha_1 + \alpha_2 - Q)(n-x)} (L_t + R_t)^{-\kappa} 1_{\overline{A}_{n, t}} \bigg| A_{n, t}, B_t = x\right] \P(A_{n, t} \big| B_t = x)  e^{-\frac{x^2}{2t}}dx\\
& = \frac{2 e^{(\alpha_1 + \alpha_2 - Q)n}}{\sqrt{2\pi}}\int_{-\infty}^n 
\E \left[  \frac{e^{-(\alpha_1 + \alpha_2 - Q)(n-x)}1_{\{\min_{s \ge 0} \beta_{L, s}^n \le 1\} \cup \{\min_{s \ge 0} \beta_{R, s}^{n-x} \le 1\}}}{(e^{\gamma n}F_{\alpha_3, \alpha_4}(n, \beta_{L, \cdot}^n) + e^{\gamma  n}F'_{\alpha_3, \alpha_4}(n-x, \beta_{R, \cdot}^{n-x}))^{\kappa}}\right] n(n-x)dx
\end{align*}

\noindent where the last line follows by dominated convergence, and is equal to
\begin{align*}
\sqrt{\frac{2}{\pi}}n e^{-(\kappa \gamma - (\alpha_1 + \alpha_2 - Q))n}\int_{0}^\infty 
\E \left[  \frac{1_{\{\min_{s \ge 0} \beta_{L, s}^n \le 1\} \cup \{\min_{s \ge 0} \beta_{R, s}^{x} \le 1\}}}{(F_{\alpha_3, \alpha_4}(n, \beta_{L, \cdot}^n) + F'_{\alpha_3, \alpha_4}(x, \beta_{R, \cdot}^{x}))^{\kappa}}\right] x e^{-(\alpha_1 + \alpha_2 - Q)x}dx
\end{align*}

\noindent so we are done.

\end{proof}

\begin{remark}
The careful reader may notice that the proof above when $\alpha_1 + \alpha_2 - Q \in (0, \kappa \gamma)$ differs slightly from that in Section \ref{subsec:proof} where one considers the event $\widetilde{A}_{n, t} = \left\{\sup_{0 \le s \le t} \widetilde{B}_s \le n\right\}$ instead of $A_{n, t}= \left\{\sup_{0 \le s \le t} B_s \le n\right\}$. The current approach, which addresses the partitioning of probability space instead of factorisation in the first place, may have the drawback that \eqref{eq:prob_rep1_2} does not give a product of two negative moments immediately but it allows for an easier side-by-side comparison with the analysis in \cite{DKRV2}.
\end{remark}

	\subsection{Proof of Theorem \ref{theo:prob_rep1}}
The infinite series representation in Lemma \ref{lem:prob_rep} is reminiscent of Riemann sums. We now explain how to obtain the simplified expressions in Theorem \ref{theo:prob_rep1}.

\begin{proof}[Proof of \eqref{eq:prob_rep2_1} and \eqref{eq:prob_rep2_3}]
We begin with $\alpha_1 + \alpha_2 - Q = 0$. Fix some $N > 0$, and without loss of generality choose a sequence of $h \to 0^+$ such that $h$ always divides both $N^{-1}$ and $N$. Then by Lemma \ref{lem:prob_rep} we have
\begin{align}\label{eq:proof_prob_rep2}
E_{\kappa}^{\gamma}(\alpha_1, \alpha_2, \alpha_3, \alpha_4)
= \sqrt{\frac{2}{\pi}} \sum_{n = 1 / Nh+ 1}^{N / h}  nh e^{-\kappa \gamma nh} \E  \left[ \left(F_{\alpha_3, \alpha_4}(nh, \beta_{\cdot}^{nh})\right)^{-\kappa} 1_{\{\min_{s > 0} \beta_s^{nh} \le h\}}\right]
+ C_N
\end{align}

\noindent for some constant $C_N > 0$ which depends on $N$ and the other parameters but not on $h$, with the property that $\lim_{N \to \infty} C_N = 0$.

Recall \eqref{eq:functional} for the definition of the random functional $F$. By Theorem \ref{theo:path_dec}, we can rewrite the sum in \eqref{eq:proof_prob_rep2} as
\begin{align*}
& \sum_{n = 1 / Nh+ 1}^{N / h}  nh e^{-\kappa \gamma nh}\int_0^{\frac{1}{n}}
\E  \left[ \left(
e^{-\gamma nh} \int_{|x| \ge 1} \frac{dM^\gamma(x)}{|x|^{4 - \gamma (\alpha_3 + \alpha_4)} |x-1|^{\gamma \alpha_3}}
\right.\right. \\
& \quad \left.\left.
+\int_{\R_{s \ge 0} \times \S^1_\theta}e^{-\gamma((nh + B_s) 1_{\{s \le T_{-nh(1-u)}\}} + (nhu + \beta_{s - T_{-nh(1-u)}}^0)1_{\{ s \ge T_{-nh(1-u)}\}} - \alpha_3G(1, e^{-s - i\theta}))}d\widehat{M}^\gamma(s,\theta)
\right)^{-\kappa}\right] du\\
& \overset{x = nh(1-u)}{=}
 \sum_{n = 1 / Nh+ 1}^{N / h}   e^{-\kappa \gamma nh}\int_{(n-1)h}^{nh}
\E  \left[ \left(
e^{-\gamma nh} \int_{|x| \ge 1} \frac{dM^\gamma(x)}{|x|^{4 - \gamma (\alpha_3 + \alpha_4)} |x-1|^{\gamma \alpha_3}}
\right.\right. \\
& \qquad \left.\left.
+\int_{\R_{s \ge 0} \times \S^1_\theta} e^{-\gamma((nh + B_s) 1_{\{s \le T_{-x}\}} + (nh - x + \beta_{s - T_{-x}}^0)1_{\{ s \ge T_{-x}\}} - \alpha_3G(1, e^{-s - i\theta}))}d\widehat{M}^\gamma(s,\theta)
\right)^{-\kappa}\right] dx\\
& = (1+o(1)) \int_{1 / N}^{N}  e^{-\kappa \gamma x}
\E  \left[ \left(F_{\alpha_3, \alpha_4}(x, \widetilde{\beta}_{\cdot}^x)\right)^{-\kappa}\right] dx
\end{align*}

\noindent where the $o(1)$ error is with respect to $h \to 0^+$ and comes from the fact that
\begin{align*}
e^{-\gamma nh} = (1+o(1))e^{-\gamma x},
\qquad e^{-\gamma (nh - x)} = (1+o(1))
\end{align*}

\noindent uniformly in $h > 0$ and $n \in \N$ for all $x \in [(n-1)h, nh]$. The desired formula \eqref{eq:prob_rep2_1} is recovered by sending $h \to 0^+$ and $N \to \infty$. The proof of  \eqref{eq:prob_rep2_3} is similar.
\end{proof}

The case where $\alpha_1 + \alpha_2 - Q \in (0, \kappa \gamma)$ is slightly more involved and the following elementary formula will be useful.
\begin{lemma} \label{lem:-moments}
Fix $\kappa, \gamma, \lambda > 0$ such that $\lambda < \kappa \gamma$. Let $X, Y$ be independent non-negative random variables and $T$ an independent $\mathrm{Exp}(\lambda)$ random variable. Provided that all the moments below exist, we have
\begin{align}\label{eq:-moments}
\E\left[\left(X+e^{-\gamma T} Y\right)^{-\kappa}\right]
= \frac{\lambda}{\gamma}B\left(\frac{\lambda}{\gamma}, \kappa - \frac{\lambda}{\gamma}\right)\E\left[X^{-(\kappa - \frac{\lambda}{\gamma})}\right]
\E\left[Y^{-\frac{\lambda}{\gamma}}\right].
\end{align}

\noindent where $B(x, y) = \frac{\Gamma(x) \Gamma(y)}{\Gamma(x+y)}$ is the beta function.
\end{lemma}

The proof of Lemma \ref{lem:-moments} follows from the same change-of-variable argument in \eqref{eq:lim_b} and is skipped here. For a sanity check one may quickly verify that both the LHS and RHS of \eqref{eq:-moments} converge to $\E[X^{-\kappa}]$ as $\lambda / \gamma \to 0$.

\begin{proof}[Proof of \eqref{eq:prob_rep2_2}]
Our starting point is \eqref{eq:prob_rep1_2} from Lemma \ref{lem:prob_rep}. It is clear that
\begin{align*}
& \E \left[ \frac{1_{\{\min_{s > 0} \beta_{L, s}^{nh} \le h\}\cup \{\min_{s > 0} \beta_{R, s}^{\mathcal{T}} \le h\}}}{\left(F_{\alpha_3, \alpha_4}(nh, \beta_{L, \cdot}^{nh}) + F'_{\alpha_2, \alpha_1}(\mathcal{T}, \beta_{R, \cdot}^{\mathcal{T}})\right)^{\kappa}}
\right]
= \E \left[ \frac{1_{\{\min_{s > 0} \beta_{L, s}^{nh} \le h\}}}{\left(F_{\alpha_3, \alpha_4}(nh, \beta_{L, \cdot}^{nh}) + F'_{\alpha_2, \alpha_1}(\mathcal{T}, \beta_{R, \cdot}^{\mathcal{T}})\right)^{\kappa}}
\right]\\
& \qquad + \E \left[ \frac{1_{\{\min_{s > 0} \beta_{R, s}^{\mathcal{T}} \le h\}}}{\left(F_{\alpha_3, \alpha_4}(nh, \beta_{L, \cdot}^{nh}) + F'_{\alpha_2, \alpha_1}(\mathcal{T}, \beta_{R, \cdot}^{\mathcal{T}})\right)^{\kappa}}
\right]
- \E \left[ \frac{1_{\{\min_{s > 0} \beta_{L, s}^{nh} \le h\}\cap \{\min_{s > 0} \beta_{R, s}^{\mathcal{T}} \le h\}}}{\left(F_{\alpha_3, \alpha_4}(nh, \beta_{L, \cdot}^{nh}) + F'_{\alpha_2, \alpha_1}(\mathcal{T}, \beta_{R, \cdot}^{\mathcal{T}})\right)^{\kappa}}
\right]
\end{align*}

\noindent where the last term is $O(h^2)$ and may be safely ignored. Arguing as before, we see that
\begin{align}
\notag & \sqrt{\frac{2}{\pi}} \sum_{n=1}^{\infty} \frac{ nh e^{-(\kappa \gamma - (\alpha_1 + \alpha_2 - Q))nh}}{(\alpha_1 + \alpha_2 - Q)^2 } 
\E \left[ \frac{1_{\{\min_{s > 0} \beta_{L, s}^{nh} \le h\}}}{\left(F_{\alpha_3, \alpha_4}(nh, \beta_{L, \cdot}^{nh}) + F'_{\alpha_2, \alpha_1}(\mathcal{T}, \beta_{R, \cdot}^{\mathcal{T}})\right)^{\kappa}}
\right]\\
\label{eq:proof_rep2_2a}
& = \frac{\sqrt{2/\pi}}{(\alpha_1 + \alpha_2 - Q)^2 (\kappa \gamma - (\alpha_1 + \alpha_2 - Q))}
\E \left[ \left(F_{\alpha_3, \alpha_4}(\tau, \widetilde{\beta}_{L, \cdot}^{\tau}) + F'_{\alpha_2, \alpha_1}(\mathcal{T}, \beta_{R, \cdot}^{\mathcal{T}})\right)^{-\kappa}
\right] + o(1)
\end{align}

\noindent where $\tau \sim \mathrm{Exp}(\kappa \gamma - (\alpha_1 + \alpha_2 - Q))$ and $\mathcal{T} \sim \mathrm{Gamma}(2, \alpha_1 + \alpha_2 - Q)$. Recall that if $U$ is an independent $\mathrm{Uniform}[0,1]$ random variable, then $(\mathcal{T}_1, \mathcal{T}_2) : = (\mathcal{T}U, \mathcal{T}(1-U))$ is a pair of independent $\mathrm{Exp}(\alpha_1 + \alpha_2 - Q)$ random variables. Combining this fact with Theorem \ref{theo:path_dec}, we obtain
\begin{align*}
F'_{\alpha_2, \alpha_1}(\mathcal{T}, \beta_{R, \cdot}^{\mathcal{T}})
\overset{d}{=}e^{-\gamma \mathcal{T}_1} F'_{\alpha_2, \alpha_1}(\mathcal{T}_2, \widetilde{\beta}_{R, \cdot}^{\mathcal{T}_2})
\end{align*}

\noindent and we can rewrite the expectation in \eqref{eq:proof_rep2_2a} as
\begin{align*}
&  \E \left[ \left(F_{\alpha_3, \alpha_4}(\tau, \widetilde{\beta}_{L, \cdot}^{\tau}) + e^{-\gamma \mathcal{T}_1} F'_{\alpha_2, \alpha_1}(\mathcal{T}_2, \widetilde{\beta}_{R, \cdot}^{\mathcal{T}_2})\right)^{-\kappa}\right].
\end{align*}

\noindent Similarly, if we let $\tau_1, \tau_2$ be independent $\mathrm{Exp}(\kappa \gamma - (\alpha_1 + \alpha_2 - Q))$, then
\begin{align}
\notag &  \sqrt{\frac{2}{\pi}}\sum_{n=1}^{\infty} \frac{ nh e^{-(\kappa \gamma - (\alpha_1 + \alpha_2 - Q))nh} }{(\alpha_1 + \alpha_2 - Q)^2} 
\E \left[ \frac{1_{\{\min_{s > 0} \beta_{R, s}^{\mathcal{T}} \le h\}}}{\left(F_{\alpha_3, \alpha_4}(nh, \beta_{L, \cdot}^{nh}) + F'_{\alpha_2, \alpha_1}(\mathcal{T}, \beta_{R, \cdot}^{\mathcal{T}})\right)^{\kappa}}
\right]\\
\label{eq:proof_rep2_2b}
& = \frac{\sqrt{2/\pi}}{(\alpha_1 + \alpha_2 - Q)(\kappa \gamma - (\alpha_1 + \alpha_2 - Q))^2}\E \left[ \left(e^{-\gamma\tau_1}F_{\alpha_3, \alpha_4}(\tau_2, \widetilde{\beta}_{L, \cdot}^{\tau_2}) + F'_{\alpha_2, \alpha_1}(\mathcal{T}_2, \widetilde{\beta}_{R, \cdot}^{\mathcal{T}_2})\right)^{-\kappa}
\right] + o(1).
\end{align}

\noindent The claim then follows by sending $h \to 0^+$ and applying Lemma \ref{lem:-moments} to \eqref{eq:proof_rep2_2a} and \eqref{eq:proof_rep2_2b}. 
\end{proof}

\section{Fusion in boundary Liouville Conformal Field Theory}
\subsection{Boundary Liouville Conformal Field Theory}
Boundary LCFT is LCFT on proper simply connected domains $D\subset\C$. We start by a brief review of the theory and refer to \cite{huang} for details. Like LCFT on the sphere, the theory is conformally invariant, so by the Riemann uniformisation theorem, it is enough to study it on the upper-half plane $\H:=\{\Im z>0\}$ (the unit disc $\D$ is also a common choice) equipped with some background metric $g$. In this context, the Liouville action with boundary term is given by\footnote{As in the sphere case, we omit the Ricci and geodesic curvature terms.}
\begin{equation}
\label{eq:boundary_liouville_action}
S_\mathrm{L}(X,g):=\frac{1}{4\pi}\int_\H\left(|\nabla X|^2+4\pi\mu e^{\gamma X}g(z)\right)d^2z+\mu_\partial\int_\R e^{\frac{\gamma}{2}X}g(x)^{1/2}dx
\end{equation}
where $\mu_\partial>0$ is the boundary cosmological constant. One recognises the Dirichlet energy in the first term of the action, giving rise to a GFF which we take to have Neumann boundary conditions. The GFF is weighted by its bulk GMC mass $M^\gamma(\H)$ and its boundary GMC mass $M^\gamma_\partial(\R)$, where the boundary GMC is formally
\[dM^\gamma_\partial(x)=e^{\frac{\gamma}{2}X(x)-\frac{\gamma^2}{8}\E[X(x)^2]}g(x)^{1/2}dx\]
and is obtained via a regularisation of the field using semi-circle averages.

 As in the sphere case, the observables are the vertex operators $V_\alpha(z)$ for insertions $z\in\H$. The main difference is that one can consider insertions on the boundary, which we formally write 
 \[B_\beta(x):=e^{\frac{\beta}{2}X(x)}\]
 for $x\in\R$ and $\beta$ in a range to be determined. The correlation functions $\langle\prod_{i=1}^NV_{\alpha_i}(z_i)\prod_{j=1}^MB_{\beta_j}(x_j)\rangle$ exist iff the Seiberg bounds are satisfied, which in this context are given by
\begin{equation}
\begin{aligned} 
 &\sigma:=\sum_{i=1}^N\frac{\alpha_i}{Q}+\sum_{j=1}^M\frac{\beta_j}{2Q}-1>0\\
 &\forall i,\;\alpha_i<Q\\
 &\forall j,\;\beta_j<Q
 \end{aligned}
 \end{equation}
 If these are satisfied, the correlation function has the following form\footnote{We chose the prefactor 2 so that the asymptotic behaviour of the bulk 1-point function with $\mu=0$ coincides with that of \cite{fateev} equation (2.24).} \cite{huang}:
\begin{equation}
\label{eq:correl_disc}
\left\langle\prod_{i=1}^NV_{\alpha_i}(z_i)\prod_{j=1}^MB_{\beta_j}(x_j)\right\rangle=2e^{C(\z,\x)}\int_\R e^{Q\sigma c}\E\left[\exp\left(-\mu e^{\gamma c}\int_\H e^{\gamma H}dM^\gamma-\mu_\partial e^{\frac{\gamma}{2}c}\int_\R e^{\frac{\gamma}{2}H}dM^\gamma_\partial\right)\right]dc
\end{equation}
where $H$ and $C(\z,\x)$ are the functions defined by
\begin{equation}
\begin{aligned}
H=&\sum_{i=1}^N\alpha_iG(z_i,\cdot)+\sum_{j=1}^M\frac{\beta_j}{2}G(x_j,\cdot)\\
C(\z,\x)=&\sum_{i<i'}\alpha_i\alpha_{i'}G(z_i,z_i')+\sum_{i,j}\frac{\alpha_i\beta_j}{2}G(z_i,x_j)+\sum_{j<j'}\frac{\beta_j\beta_{j'}}{4}G(x_j,x_j')
\end{aligned}
\end{equation}
with $G$ being Green's function with Neumann boundary conditions on $(\H,g)$. Notice that the usual change of variable $u=e^{\gamma c}$ does not give a nicer expression in this case since the exponential term in the expectation is quadratic in $e^{\frac{\gamma}{2}c}$.

Correlation functions are conformally covariant, and if $\psi:\H\to\H$ is a M\"obius transformation, then (recall that $\Delta_{\alpha}=\frac{\alpha}{2}(Q-\frac{\alpha}{2})$)
\[\left\langle\prod_{i=1}^NV_{\alpha_i}(\psi(z_i))\prod_{j=1}^MB_{\beta_j}(\psi(x_j))\right\rangle=\prod_{i=1}^N|\psi'(z_i)|^{-2\Delta_{\alpha_i}}\prod_{j=1}^M|\psi'(x_j)|^{-\Delta_{\beta_j}}\left\langle\prod_{i=1}^NV_{\alpha_i}(z_i)\prod_{j=1}^MB_{\beta_j}(x_j)\right\rangle\]

M\"obius transforms of $\H$ have three real parameters, so when the location of the insertions have less than (or exactly) three real parameters, the correlation functions are determined by conformal invariance, and we have the following structure constants
\begin{enumerate}
\item Bulk-boundary two-point function
\begin{equation}
\langle V_\alpha(z)B_\beta(x)\rangle=\frac{R(\alpha,\beta)}{|z-\bar{z}|^{2\Delta_\alpha-\Delta_\beta}|z-x|^{2\Delta_\beta}}
\end{equation}
As a special case of this equation for $\beta=0$, we have the bulk one-point function
\begin{equation}
\label{eq:bulk_one_point}
\langle V_\alpha(z)\rangle=\frac{U(\alpha)}{|z-\bar{z}|^{2\Delta_\alpha}}
\end{equation}
\item Boundary three-point function
\begin{equation}
\label{eq:boundary_three_point}
\langle B_{\beta_1}(x_1)B_{\beta_2}(x_2)B_{\beta_3}(x_3)\rangle=\frac{c(\beta_1,\beta_2,\beta_3)}{|x_1-x_2|^{\Delta_{\beta_1}+\Delta_{\beta_2}-\Delta_{\beta_3}}|x_2-x_3|^{\Delta_{\beta_2}+\Delta_{\beta_3}-\Delta_{\beta_1}}|x_3-x_1|^{\Delta_{\beta_3}+\Delta_{\beta_1}-\Delta_{\beta_2}}}
\end{equation}
\end{enumerate}
\begin{remark}
There is also a definition for a boundary two-point function, which we omit here since we will not be needing it for the purpose of this paper. Let us just mention that this object is to the reflection coefficient of \cite{KRV2} what the boundary three-point function is to the DOZZ formula.
\end{remark}
The above structure constants are to be understood as meromorphic functions of the parameters and they arise naturally in the bootstrap formalism. Physicists have conjectured exact formulae for the values of these structure constants \cite{fateev,ponsot}, and there are works in progress by Gwynne and Remy establishing the validity of \eqref{eq:bulk_one_point} and Remy and Zhu addressing \eqref{eq:boundary_three_point}. 

\subsection{Main results}
The cases we treat are the fusion on two boundary-insertions, the absorption of a bulk-insertion on the boundary and the fusion of two bulk-insertions.
\begin{theorem}[Boundary four-point]
\label{thm:boundary_four}
Let $(\beta_1,\beta_2,\beta_3,\beta_4)$ satisfying the Seiberg bounds and suppose that $\beta_3+\beta_4>Q$. Then the following asymptotic holds:
\begin{enumerate}
\item \emph{Supercritical case}

If $\beta_1+\beta_2>Q$, then
\begin{equation}
\langle B_{\beta_1}(0)B_{\beta_2}(x)B_{\beta_3}(1)B_{\beta_4}(\infty)\rangle\underset{x\to0}{\sim}\frac{1}{4\sqrt{\pi}}\frac{|x|^{\frac{Q^2}{4}-\Delta_{\beta_1}-\Delta_{\beta_2}}}{\log^{3/2}\frac{1}{|x|}}\partial_3c(\beta_1,\beta_2,Q)\partial_1c(Q,\beta_3,\beta_4)
\end{equation}
\item \emph{Critical case}

If $\beta_1+\beta_2=Q$, then
\begin{equation}
\label{eq:fusion_boundary_critical}
\langle B_{\beta_1}(0)B_{\beta_2}(x)B_{\beta_3}(1)B_{\beta_4}(\infty)\rangle\underset{x\to0}{\sim}-\frac{1}{\sqrt{\pi}}\frac{|x|^{-\frac{1}{2}\beta_1\beta_2}}{\log^{1/2}\frac{1}{|x|}}\partial_1c(Q,\beta_3,\beta_4)
\end{equation}
\item \emph{Subcritical case}

If $\beta_1+\beta_2<Q$, then
\begin{equation}
\langle B_{\beta_1}(0)B_{\beta_2}(x)B_{\beta_3}(1)B_{\beta_4}(\infty)\rangle\underset{x\to0}{\sim}|x|^{-\frac{1}{2}\beta_1\beta_2}c(\beta_1+\beta_2,\beta_3,\beta_4)
\end{equation}
\end{enumerate}
\end{theorem}

The next theorem is about the fusion in the bulk two-point function. 
\begin{theorem}[Bulk two-point: Fusion]
\label{thm:bulk_two_fusion}
Let $(\alpha_1,\alpha_2,\beta)$ satisfying the Seiberg bounds. Then the following asymptotics hold:
\begin{enumerate}
\item If $\beta=0$, then
\begin{equation}
\langle V_{\alpha_1}(i)V_{\alpha_2}(i+z)\rangle\underset{z\to 0}{\sim}-\frac{2^{-\alpha_1\alpha_2}}{\sqrt{2\pi}}\frac{|z|^{2(\frac{Q^2}{4}-\Delta_{\alpha_1}-\Delta_{\alpha_2})}}{\log^{1/2}\frac{1}{|z|}}\partial_3C_\gamma(\alpha_1,\alpha_2,Q)
\end{equation}
\item If $\beta>0$, then
\begin{enumerate}
\item\emph{Supercritical case}

If $\alpha_1+\alpha_2>Q$, then
\begin{equation}
\langle V_{\alpha_1}(i)V_{\alpha_2}(i+z)B_\beta(0)\rangle\underset{z\to0}{\sim}\frac{2^{\Delta_\beta-\frac{Q^2}{2}-\alpha_1\alpha_2}}{4\sqrt{2\pi}}\frac{|z|^{2(\frac{Q^2}{4}-\Delta_{\alpha_1}-\Delta_{\alpha_2})}}{\log^{3/2}\frac{1}{|z|}}\partial_3C_\gamma(\alpha_1,\alpha_2,Q)\partial_1R(Q,\beta)
\end{equation}
\item \emph{Critical case}

If $\alpha_1+\alpha_2=Q$, then
\begin{equation}
\langle V_{\alpha_1}(i)V_{\alpha_2}(i+z)B_\beta(0)\rangle\underset{z\to0}{\sim}-\frac{2^{\Delta_\beta-\frac{Q^2}{2}-\alpha_1\alpha_2}}{\sqrt{2\pi}}\frac{|z|^{-\alpha_1\alpha_2}}{\log^{1/2}\frac{1}{|z|}}\partial_1R(Q,\beta)
\end{equation}
\item\emph{Subcritical case}

If $\alpha_1+\alpha_2<Q$, then
\begin{equation}
\langle V_{\alpha_1}(i)V_{\alpha_2}(i+z)B_\beta(0)\rangle\underset{z\to0}{\sim}2^{\Delta_\beta-\frac{Q^2}{2}-\alpha_1\alpha_2}|z|^{-\alpha_1\alpha_2}R(\alpha_1+\alpha_2,\beta)
\end{equation}
\end{enumerate}
\end{enumerate}
\end{theorem}
Another interesting limit of the bulk two-point function is sending one insertion to the boundary.
\begin{theorem}[Bulk two-point: Absorption]
\label{thm:bulk_two_absorption}
Let $(\alpha_1,\alpha_2)$ satisfying the Seiberg bounds, and suppose $\alpha_1>\frac{Q}{2}$. Then the following asymptotic holds:
\begin{enumerate}
\item\emph{Supercritical case}

 If $\alpha_2>\frac{Q}{2}$, then
\begin{equation}
\langle V_{\alpha_1}(i)V_{\alpha_2}(z)\rangle\underset{z\to0}{\sim}\frac{2^{2(\frac{Q^2}{4}-\Delta_{\alpha_1}-\Delta_{\alpha_2})}}{4\sqrt{\pi}}\frac{|z|^{(\alpha_2-\frac{Q}{2})^2}}{\log^{3/2}\frac{1}{|z|}}\partial_2R(\alpha_1,Q)\partial_2R(\alpha_2,Q)
\end{equation}
\item\emph{Critical case}

If $\alpha_2=\frac{Q}{2}$, then
\begin{equation}
\langle V_{\alpha_1}(i)V_{\alpha_2}(z)\rangle\underset{z\to0}{\sim}-\frac{2^{\frac{Q^2}{2}-2\Delta_{\alpha_1}}}{\sqrt{\pi}\log^{1/2}\frac{1}{|z|}}\partial_2R(\alpha_1,Q)
\end{equation}
\item\emph{Subcritical case}

If $\alpha_2<\frac{Q}{2}$, then
\begin{equation}
\langle V_{\alpha_1}(i)V_{\alpha_2}(z)\rangle\underset{z\to0}{\sim}\frac{R(\alpha_1,2\alpha_2)}{2^{2\Delta_{\alpha_1}-\Delta_{2\alpha_2}}}
\end{equation}
\end{enumerate}
\end{theorem}
We now turn to the bulk-boundary three-point function $\langle V_\alpha(z)B_{\beta_1}(0)B_{\beta_2}(\infty)\rangle$. There is not much to say about the merging of the bulk insertion with a boundary insertion since for all $r>0$ and $\theta\in(0,\pi)$, the correlation function $\langle V_\alpha(re^{i\theta})B_{\beta_1}(0)B_{\beta_2}(\infty)\rangle$ is deduced from $\langle V_\alpha(e^{i\theta})B_{\beta_1}(0)B_{\beta_2}(\infty)\rangle$ by scaling. The non-trivial parameter we can vary is $\theta$, and the limit $\theta\to0$ corresponds to the absorption of an bulk insertion on a boundary point which is not an insertion. Thus we will study the correlation function $\langle V_\alpha(z)B_{\beta_1}(1)B_{\beta_2}(\infty)\rangle$ in the limit $z\to0$. Notice that by M\"obius invariance, this is the same as studying the function $\langle V_\alpha(i)B_{\beta_1}(0)B_{\beta_2}(x)\rangle$ in the limit $x\to0$, i.e. merging the two boundary insertions.

\begin{theorem}[Bulk-boundary three-point]
\label{thm:bulk_boundary}
Let $(\alpha,\beta_1,\beta_2)$ satisfying the Seiberg bounds, and assume that $\beta_1+\beta_2>\frac{Q}{2}$. Then the following asymptotic holds
\begin{enumerate}
\item\emph{Supercritical case}

If $\alpha>\frac{Q}{2}$, then
\begin{equation}
\langle V_\alpha(z)B_{\beta_1}(1)B_{\beta_2}(\infty)\rangle\underset{z\to0}{\sim}\frac{2^{\frac{Q^2}{4}-2\Delta_\alpha}}{4\sqrt{\pi}}\frac{|z|^{(\alpha-\frac{Q}{2})^2}}{\log^{3/2}\frac{1}{|z|}}\partial_2R(\alpha,Q)\partial_1c(Q,\beta_1,\beta_2)
\end{equation}
\item\emph{Critical case}

If $\alpha=\frac{Q}{2}$, then
\begin{equation}
\langle V_\alpha(z)B_{\beta_1}(1)B_{\beta_2}(\infty)\rangle\underset{z\to0}{\sim}-\frac{1}{\sqrt{\pi}\log^{1/2}\frac{1}{|z|}}\partial_1c(Q,\beta_1,\beta_2)
\end{equation}
\item\emph{Subcritical case}

If $\alpha<\frac{Q}{2}$, then
\begin{equation}
\langle V_\alpha(z)V_{\beta_1}(1)V_{\beta_2}(\infty)\rangle\underset{z\to0}{\sim}c(2\alpha,\beta_1,\beta_2)
\end{equation}
\end{enumerate}
\end{theorem}

\begin{remark}
More generally, the fusion rules in the supercritical case are the following:
\begin{enumerate}
\item Fusion of boundary-boundary $(\beta_1,\beta_2)$-insertions produces a boundary three-point function $\partial_3c(\beta_1,\beta_2,Q)$.
\item Absorption of a bulk $\alpha$-insertion produces a bulk-boundary function $\partial_2R(\alpha,Q)$.
\item Fusion of bulk-bulk $(\alpha_1,\alpha_2)$-insertions produces a DOZZ formula $\partial_3C_\gamma(\alpha_1,\alpha_2,Q)$.
\end{enumerate}
This rule, as well as the rate functions of the above theorems, can be used to compute the asymptotic behaviour of all correlation functions upon fusion of insertions, and express the limit with a lower order correlation function.
\end{remark}

As such, we haven't said anything about the fusion of bulk-boundary insertions. This is because it can be seen as a two-step procedure of first absorbing the bulk insertion into the boundary and then merging the boundary insertions. Hence the operation does not produce a structure constant. As an example, consider the correlation function $\langle V_\alpha(z)B_{\beta_1}(0)B_{\beta_2}(1)B_{\beta_3}(\infty)\rangle$ in the limit $z\to0$, for $(\alpha,\beta_1,\beta_2,\beta_3)$ satisfying the Seiberg bounds, and suppose that both $\beta_3+\beta_4>Q$ and $2\alpha+\beta_1>Q$, so that we are in the supercritical case. Then the asymptotic is given by
\begin{equation}
\langle V_\alpha(z)B_{\beta_1}(0)B_{\beta_2}(1)B_{\beta_3}(\infty)\rangle\underset{z\to0}{\sim}\frac{1}{4\sqrt{\pi}}\frac{|z|^{(\alpha-\frac{Q}{2})^2-\alpha\beta_1}}{\log^{3/2}\frac{1}{|z|}}\frac{\partial}{\partial\beta}\langle V_\alpha(i)B_{\beta_1}(0)B_\beta(\infty)\rangle_{|\beta=Q}\partial_1c(Q,\beta_2,\beta_3)
\end{equation}

\begin{remark}
Even though the correlation functions can no longer be expressed in terms of negative moments of GMC (unless $\mu \mu_{\partial} = 0$), it is still possible to give probabilistic representations of the renormalised constants in the aforementioned theorems by performing the same partitioning-of-probability-space procedure on
\begin{align*}
\E\left[ \exp\left(-\mu e^{\gamma c}\int_\H e^{\gamma H}dM^\gamma-\mu_\partial e^{\frac{\gamma}{2}c}\int_\R e^{\frac{\gamma}{2}H}dM^\gamma_\partial\right) \right]
\end{align*}

\noindent as we did in Section \ref{sec:proof_gmc}. We omit the details here.
\end{remark}

We now turn to proving Theorems \ref{thm:boundary_four}, \ref{thm:bulk_two_fusion}, \ref{thm:bulk_two_absorption} and \ref{thm:bulk_boundary}. We only deal with Theorems \ref{thm:boundary_four} and \ref{thm:bulk_two_fusion} since the other cases are similar. 

Subcritical cases follow from dominated convergence so we won't treat them. The rest of the proofs are very similar to that of Theorem \ref{thm:result} so we will be brief.

\begin{proof}[Proof of Theorem \ref{thm:boundary_four}]
The setting is the upper-half plane $\H$ equipped with the metric $g(z)=4|z|_+^{-4}$. We use the same procedure as for the sphere and apply the conformal change of coordinate $\psi:z\mapsto e^{-z/2}$ from the infinite strip $\mc{S}:=\R\times(0,2\pi)$ to $\H$. Then Green's function on the strip is given by the even part of Green's function on the cylinder, i.e. if $X$ is a GFF on $\R_s\times(0,2\pi)_\theta$, we have (recall \eqref{eq:decomposition_green})
\begin{equation}
\begin{aligned}
\E[X(s,\theta)X(s',\theta')]&=G(\frac{s}{2},\frac{\theta}{2},\frac{s'}{2},\frac{\theta'}{2})+G(\frac{s}{2},\frac{\theta}{2},\frac{s'}{2},-\frac{\theta'}{2})\\&=(|s|\wedge|s'|)1_{ss'\geq0}+H(\frac{s}{2},\frac{\theta}{2},\frac{s'}{2},\frac{\theta'}{2})+H(\frac{s}{2},\frac{\theta}{2},\frac{s'}{2},-\frac{\theta'}{2})\\
&=(|s|\wedge|s'|)1_{ss'\geq0}+G(0,0,\frac{s'-s}{2},\frac{\theta'-\theta}{2})+G(0,0,\frac{s'-s}{2},\frac{\theta'+\theta}{2})
\end{aligned}
\end{equation}
Hence the field decomposes into the independent sum $X=B+Y$ where $(B_s)_{s\in\R}$ is standard two-sided Brownian motion and $Y$ is a $\log$-correlated field whose covariance kernel is given by the sum of $G$ functions on the right-hand side of the previous equation. It is also clear from the definition that the law of $Y$ is translation invariant. The pullback measure of $g$ on the strip is $g_\psi(s,\theta)=e^{-|s|}$ so we can take the GMC measure of $Y$ with respect to Lebesgue measure on $\mc{S}$ and take the drifted process $B_s-\frac{Q}{2}|s|$ for the radial part of the GFF.

First we have to explain how to make sense of boundary (derivative) $Q$-insertions. A boundary insertion with momentum $\beta$ at $\infty$ (on the strip) amounts in adding a positive drift $\frac{\beta}{2}$ to the radial process (on the positive real line), so the total drift vanishes when $\beta=Q$. For $t>0$, define $\H_t:=\H\setminus(e^{-t/2}\D)$ (resp. $\R_t:=\R\setminus(-e^{-t/2},e^{-t/2})$) and $\langle B_Q(0)B_{\beta_2}(1)B_{\beta_3}(\infty)\rangle_t$ the correlation function where we integrate the bulk (resp. boundary) GMC measure of \eqref{eq:correl_disc} on $\H_t$ (resp. $\R_t$) instead of $\H$ (resp. $\R$). Viewed in the strip, this is the same as taking $\mc{S}_t:=(-\infty,t)\times(0,2\pi)$ and $(-\infty,t)\times\{0,2\pi\}$ as domains of integration for the bulk and boundary measures. 

For fixed $b>0$, we have
\begin{equation}
\begin{aligned}
&\P\left(\underset{0\leq s\leq t}{\sup}B_s\leq b\right)\underset{t\to\infty}{\sim}\sqrt{\frac{2}{\pi t}}b\\
&\P\left(\underset{t\geq0}{\sup}B_s+\frac{1}{2}(\beta-Q)s\leq b\right)\underset{\beta\to Q^-}{\sim}(Q-\beta)b
\end{aligned}
\end{equation}
 so by previous arguments we have
\[\underset{t\to\infty}{\lim}\sqrt{\frac{\pi t}{2}}\langle B_Q(0)B_{\beta_2}(1)B_{\beta_3}(\infty)\rangle_t=\underset{\beta\to Q^-}{\lim}\frac{1}{Q-\beta}\langle B_{\beta}(0)B_{\beta_2}(1)B_{\beta_3}(\infty)\rangle=-\partial_1c(Q,\beta_2,\beta_3)\]
The critical case \eqref{eq:fusion_boundary_critical} follows easily from this equality.

Now we turn to the supercritical case. We write $t:=2\log\frac{1}{|x|}$. The radial process has a positive drift $\frac{1}{2}(\beta_1+\beta_2-Q)$ in $(0,t)$, which we kill by Cameron-Martin's theorem (recall \eqref{eq:rn}, yielding the Radon-Nikodym derivative $e^{\frac{1}{2}(\beta_1+\beta_2-Q)B_t-\frac{1}{8}(\beta_1+\beta_2-Q)^2t}$. This accounts for the polynomial rate in $|x|$. 

Similarly as in Figure \ref{fig:field_decomposition}, we condition on value of the process at time $t$ and introduce $B_t=\sqrt{t}\delta$ with $\delta\sim\mc{N}(0,1)$ independent of everything. Thus the process in $[0,t]$ is the sum of a random drift $\frac{\delta}{\sqrt{t}}$ and an independent Brownian bridge in $[0,t]$ (see Figure \ref{fig:field_decomposition_disc}). Conditioning the Brownian bridge in $(0,t)$ to stay below $b$, we get a contribution of $\sqrt{\frac{2}{\pi}}t^{-3/2}=\frac{1}{2\sqrt{2\pi}\log^{3/2}\frac{1}{|x|}}$. Taking $t\to\infty$ then $b\to\infty$, the limiting integral on the left is a strip with a $\beta_4$-insertions at $-\infty$, a $\beta_3$-insertion at 0 and a (derivative) $Q$-insertion at $+\infty$ (see Figure \ref{fig:field_decomposition_disc}), hence the limit is $-\frac{1}{2}\partial_1c(Q,\beta_3,\beta_4)$ (recall the prefactor 2 in the definition of \eqref{eq:correl_disc}). Similarly the limiting integral on the left is $-\frac{1}{2}\partial_1c(\beta_1,\beta_2,Q)$, yielding the result.

\begin{figure}[hbtp]
\centering
\includegraphics[scale=0.6]{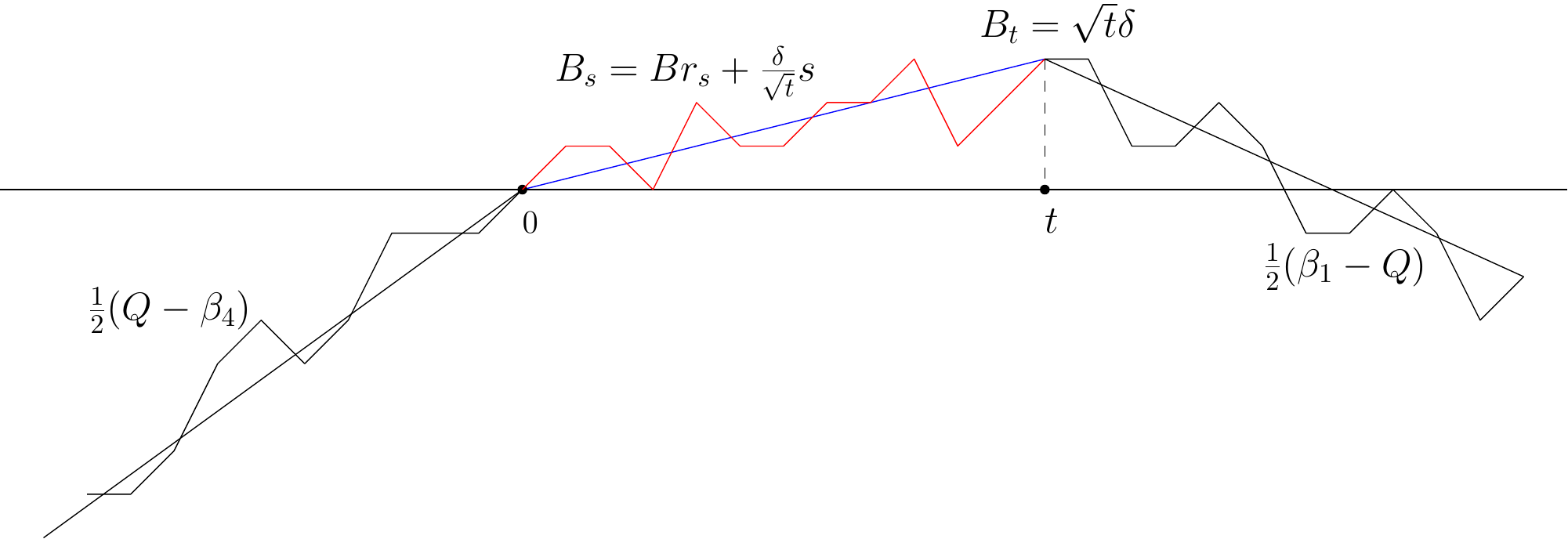}
\caption{\label{fig:field_decomposition_disc} The radial process on the strip in $[0,t]$ is the sum of a Brownian bridge (red) and a random independent drift (blue).}
\end{figure}

\end{proof}

\begin{proof}[Proof of Theorem \ref{thm:bulk_two_fusion}]
In this proof, we use the flat disc $(\D,dz)$ as set-up, which is mapped to the semi-infinite cylinder $\mc{C}_+=\R_+\times\S^1$ equipped with the metric $g(s,\theta)=e^{-2s}$ under the conformal transformation $z\mapsto e^{-z}$. So the GFF decomposes as the sum of a drifted Brownian motion $(B_s-Qs)_{s\geq0}$ and an independent lateral noise $Y$ from which we take the GMC measure with respect to Lebesgue measure.

We treat the case $\beta>0$ and $\alpha_1+\alpha_2>Q$, the others being similar. Let $t:=\log\frac{1}{|z|}$. With the presence of the insertions, the radial part has a positive drift $\alpha_1+\alpha_2-Q$ in $(0,t)$ and negative drift $\alpha_1-Q$ in $(t,\infty)$. Killing the drift in $(0,t)$ with Cameron-Martin's theorem gives the exponent in $|z|$. Conditioning on the value of $B_t=\sqrt{t}\delta$ and conditioning the Brownian bridge not to exceed some $b>0$ gives a prefactor of $\sqrt{\frac{2}{\pi}}t^{-3/2}$. Taking $t\to\infty$ then $b\to\infty$, we find that the integral on the right is an infinite cylinder with insertions $(\alpha_1,\alpha_2,Q)$ at $(+\infty,0,-\infty)$, so its value is $-\frac{1}{4}\partial_3C_\gamma(\alpha_1,\alpha_2,Q)$ (the lateral noise is close to the one used before in this region and can be dealt with using Kahane's convexity inequality). On the other hand, the integral on the left is a semi-infinite cylinder with a $Q$-insertion at $\infty$ and a $\beta$-insertion on the boundary, so its value is $-\frac{1}{4}\frac{\partial}{\partial\alpha}\langle V_\alpha(i)B_\beta(0)\rangle_{|\alpha=Q}$.
\end{proof}
\subsection{Links with random planar maps}
The above results can be interpreted with respect to the KPZ conjecture on random planar maps with the topology of the disc. For concreteness, let $\mc{T}_{n,m}$ be the set of triangulations of the disc with $n$ internal vertices and $m+2$ boundary vertices, with two marked vertices (one internal and one on the boundary). Then it is known \cite{angel} that there exists $\mu^c,\mu_\partial^c>0$ such that
\[\#\mc{T}_{n,m}\asymp e^{\mu^cn}e^{\mu^c_\partial m}m^{1/2}n^{-5/2}\]
We suppose that for a triangulation $(\t,\z,\x)$, we have conformal mapped $\t$ to $\H$ (in the manner of section \ref{subsec:rpm}) and that $\z$ is mapped to $i$ and $\x$ is mapped to $0$. For each such triangulation and $a>0$, we can construct measures $\nu^{\t,a}$ (resp. $\nu^{\t,a}_\partial$) giving mass $a^2$ (resp. $a$) to each triangle (resp. each boundary edge). Now we let $\mu:=(1+a^2)\mu^c$ and $\mu_\partial:=(1+a)\mu_\partial^c$, and sample the triangulations at random with the probability measure
\[\P_a(\t,\z,\x)=\frac{1}{Z_a}e^{-\mu|\t|}e^{-\mu_\partial\ell(\t)}\]
where $Z_a$ is the normalising constant and $\ell(\t)$ is the boundary length of $\t$. Additionally we choose the internal marked vertex uniformly in the internal vertices of $\t$ and similarly for the boundary marked vertex.

It is conjectured \cite{huang} that the pair of random measures $(\nu^{\t,a},\nu^{\t,a}_\partial)$ converges in distribution to a pair of random measures on $(\D,\partial\D)$, and the limit $(\nu,\nu_\partial)$ should be given by (some form of) LQFT on the disc. In particular, it should be the case that for all measurable sets $A\subset\H,B\subset\R$,
\begin{equation}
\begin{aligned}
&\E\left[\frac{\nu(A)}{\nu(\H)}\right]=\int_Af_{\sqrt{\frac{8}{3}},\mu^c,\mu_\partial^c}(z)d^2z\\&\E\left[\frac{\nu_\partial(B)}{\nu_\partial(\R)}\right]=\int_B\lambda_{\sqrt{\frac{8}{3}},\mu^c,\mu_\partial^c}(x)dx
\end{aligned}
\end{equation}
where we define for all $\gamma\in(0,2)$ and $\mu,\mu_\partial>0$,
\begin{equation}
\label{eq:normalise_boundary}
\begin{aligned}
&f_{\gamma,\mu,\mu_\partial}(z):=\frac{1}{Z}\langle V_\gamma(z)V_\gamma(i)B_\gamma(0)\rangle\\
&\lambda_{\gamma,\mu,\mu_\partial}(x):=\frac{1}{Z_\partial}\langle B_\gamma(x)V_\gamma(i)B_\gamma(0)\rangle
\end{aligned}
\end{equation}
where $Z,Z_\partial$ are normalising constants whose values are discussed in Appendix \ref{app:constant}. 

Similarly to the discussion of section \ref{subsec:rpm}, the result of Theorems \ref{thm:bulk_boundary} and \ref{thm:bulk_two_fusion} gives precise estimates on the expected density of vertices in different settings: internal or boundary vertices around the marked point on the boundary, internal vertices around the internal marked point, and internal vertices around the boundary.

Finally, we mention that one can formulate other conjectures involving different values of $\gamma$ (e.g. by weighting the measure $\P_a$ by some statistical mechanics model), $\mu$ and $\mu_\partial$ (e.g. by considering other types of maps).

\appendix
\section{The DOZZ formula}
\label{app:dozz}
The DOZZ formula is the expression of the 3-point correlation function on the sphere $\langle V_{\alpha_1}(0)V_{\alpha_2}(1)V_{\alpha_3}(\infty)\rangle_{\S^2}$. The formula reads
\begin{equation}
\label{eq:dozz}
\begin{aligned}
C_\gamma(\alpha_1,\alpha_2,\alpha_3)=&\left(\pi\mu\left(\frac{\gamma}{2}\right)^{2-\frac{\gamma^2}{2}}\frac{\Gamma(\gamma^2/4)}{\Gamma(1-\gamma^2/4)}\right)^{-\frac{\bbar{\alpha}-2Q}{\gamma}}\\
&\times\frac{\Upsilon'_{\frac{\gamma}{2}}(0)\Upsilon_{\frac{\gamma}{2}}(\alpha_1)\Upsilon_{\frac{\gamma}{2}}(\alpha_2)\Upsilon_{\frac{\gamma}{2}}(\alpha_3)}{\Upsilon_{\frac{\gamma}{2}}\left(\frac{\bbar{\alpha}-2Q}{2}\right)\Upsilon_{\frac{\gamma}{2}}\left(\frac{\bbar{\alpha}}{2}-\alpha_1\right)\Upsilon_{\frac{\gamma}{2}}\left(\frac{\bbar{\alpha}}{2}-\alpha_2\right)\Upsilon_{\frac{\gamma}{2}}\left(\frac{\bbar{\alpha}}{2}-\alpha_3\right)}
\end{aligned}
\end{equation}
where $\bbar{\alpha}=\alpha_1+\alpha_2+\alpha_3$ and $\Upsilon_{\frac{\gamma}{2}}$ is Zamolodchikov's special function. It has the following integral representation for $\Re z\in(0,Q)$ 
\[\log\U(z)=\int_0^\infty\left(\left(\frac{Q}{2}-z\right)^2e^{-t}-\frac{\sinh^2\left(\left(\frac{Q}{2}-z\right)\frac{t}{2}\right)}{\sinh\left(\frac{\gamma t}{4}\right)\sinh\left(\frac{t}{\gamma}\right)}\right)\frac{dt}{t}\]
and it extends holomorphically to $\C$.

It satisfies the functional relation $\U(Q-z)=\U(z)$ and it has a simple zero at $0$ if $\gamma^2\in\R\setminus\Q$\footnote{This is not really a restriction since the theory is continuous in $\gamma$}, so it has a simple zero at $Q$ too and $\U'(Q)=-\U'(0)\neq0$.

Let us introduce the notation
\[\bbar{C}_\gamma(\alpha_1,\alpha_2,\alpha_3)=\frac{\U'(0)\U(\alpha_1)\U(\alpha_2)\U(\alpha_3)}{\U(\frac{\bar{\alpha}}{2}-Q)\U(\frac{\bar{\alpha}}{2}-\alpha_1)\U(\frac{\bar{\alpha}}{2}-\alpha_2)\U(\frac{\bar{\alpha}}{2}-\alpha_3)}\]

Now we assume $\alpha_1+\alpha_2=Q$ and $\alpha_3=Q-iP$ and show the limit \eqref{eq:dozz_limit}. Then $\bar{\alpha}=2Q-iP$ and
\[\bbar{C}_\gamma(\alpha_1,Q-\alpha_1,Q-iP)=\frac{\U'(0)\U(\alpha_1)^2\U(iP)}{\U(-\frac{iP}{2})\U(\alpha_1+\frac{iP}{2})\U(\alpha_1-\frac{iP}{2})\U(\frac{iP}{2})}\underset{P\to0}{\sim}\frac{4i}{P}\]
So the product of DOZZs appearing in the bootstrap equation \eqref{eq:bootstrap} becomes
\begin{equation}
\begin{aligned}
\bbar{C}_\gamma(\alpha_1,\alpha_2,&Q-iP)\bbar{C}_\gamma(Q+iP,\alpha_3,\alpha_4)\\
&\underset{P\to0}{\sim}\frac{4\U'(0)^2\U(\alpha_3)\U(\alpha_4)}{\U(\frac{\alpha_3+\alpha_4-Q+iP}{2})\U(\frac{\alpha_3+\alpha_4-Q-iP}{2})\U(\frac{\alpha_4+Q+iP-\alpha_3}{2})\U(\frac{\alpha_3+Q+iP-\alpha_4}{2})}\\
&\underset{P\to0}{\sim}-4\partial_1\bbar{C}_\gamma(Q,\alpha_3,\alpha_4)
\end{aligned}
\end{equation}
Notice that $2Q-(\alpha_1+\alpha_2+Q-iP)\underset{P\to0}{\to}0$, so the prefactor in the DOZZ formula with Liouville momenta $(\alpha_1,Q-\alpha_1,Q-iP)$ is simply 1 in this limit. Hence
\[\underset{P\to0}{\lim}C_\gamma(\alpha_1,Q-\alpha_1,Q-iP)C_\gamma(Q+iP,\alpha_3,\alpha_4)=4\partial_1C_\gamma(Q,\alpha_3,\alpha_4)\]

\section{Conical singularities}
\label{app:conical}
Here we reproduce \cite{Bav}, Appendix B for the commodity of the reader.

We study the effect of a change of measure with respect to the Liouville field. Let $X$ be a GFF on $\S^2$ with some background metric $g$ and $dM^\gamma_g$ be the associated chaos measure (regularised in $g$). Let $\omega\in H^1_0$ be a function such that $e^{\frac{Q}{2}\omega}\in L^1(dM^\gamma_g)$. Let $\hat{g}:=e^{\omega}g$ and $dM^\gamma_{\hat{g}}$ be the chaos of $X$ regularised in $\hat{g}$. Then for all $\kappa>0$, applying successively Girsanov's theorem and conformal covariance, we find
\begin{equation}
\E\left[e^{\langle X,\frac{Q}{2}\omega\rangle_\nabla-\frac{Q^2}{8}\norm{\omega}_\nabla^2}M^\gamma_g(\S^2)^{-\kappa}\right]=\E\left[\left(\int_{\S^2}e^{\frac{\gamma Q}{2}\omega}dM^\gamma_g\right)^{-\kappa}\right]=\E\left[M^\gamma_{\hat{g}}(\S^2)^{-\kappa}\right]
\end{equation}
In particular, the vertex operator which is formally written $V_\alpha(z)=e^{\alpha X(z)-\frac{\alpha^2}{2}\E[X(z)^2]}$ is a special case of the previous setting with $\omega=\frac{2\alpha}{Q}G(z,\cdot)$. Hence, after regularising, we find that adding a vertex operator is the same as conformally multiplying the metric by Green's function, i.e. we have $\hat{g}=e^{\frac{2\alpha}{Q}G(z,\cdot)}g$. Hence the metric behaves like $|x-z|^{-\frac{2\alpha}{Q}}$ near 0 so it has a conical singularity of order $\alpha/Q$.

\begin{figure}[h!]
\centering
\includegraphics[scale=0.7]{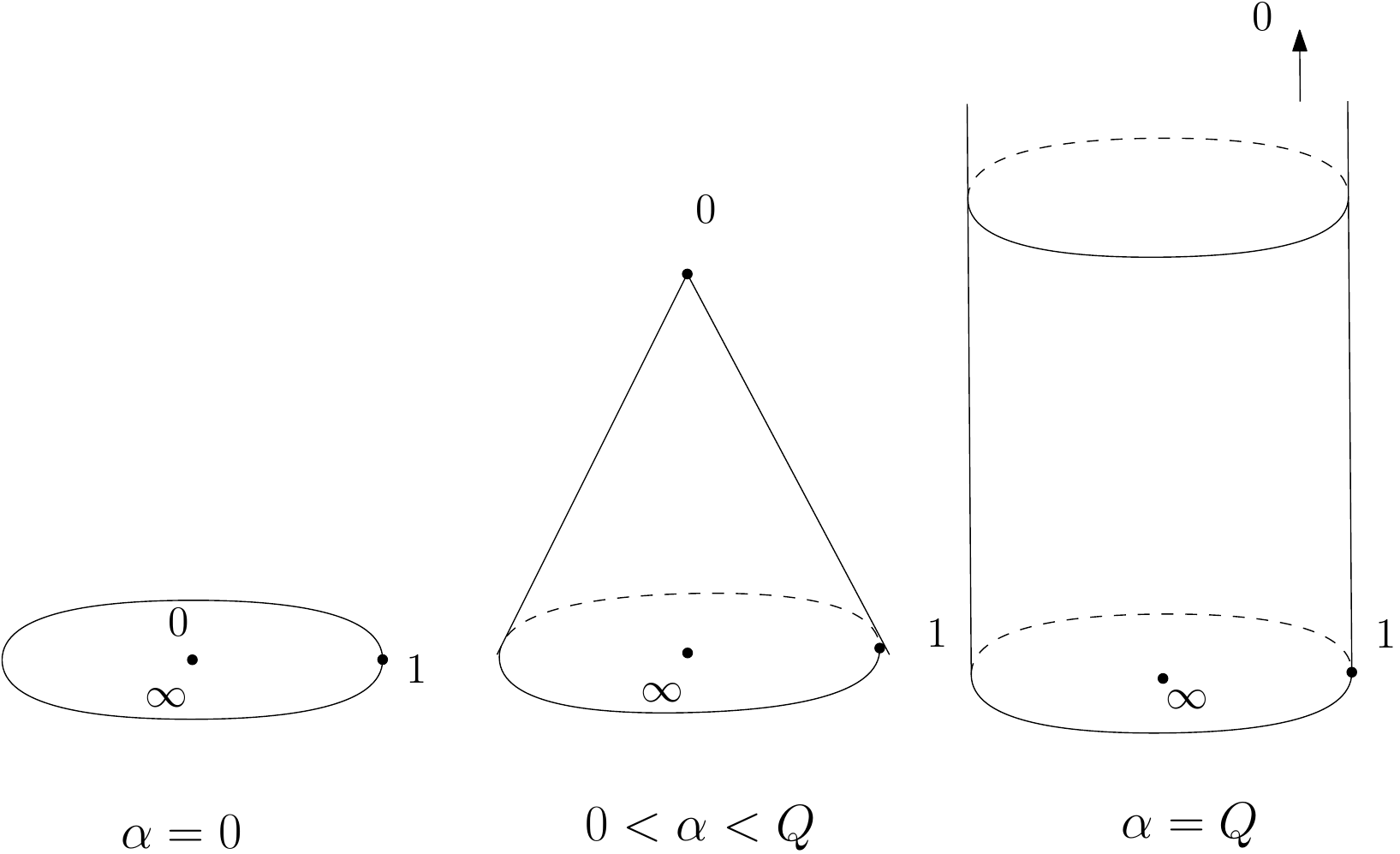}
\caption{The effect of the the vertex operator $V_\alpha(0)$ in the cr\^epe metric.}
\end{figure}

If $\alpha=Q$, the singularity is no longer integrable, so the volume is infinite and the surface has a semi-infinite cylinder. Loosely, we will refer to this situation as a cusp, even though the hyperbolic cusp has finite volume because of the extra $\log$-correction in the metric:
\[\log\hat{g}(z+h)=-2\log|h|-2\log\log\frac{1}{|h|}+O(1)\]
The reason for this abuse of terminology is that we are interested in GMC measure. Indeed, suppose $z=0$ in the sphere coordinates. By conformal covariance, if we use the cylinder coordinates, the $\log$-correction term is the same as shifting the radial part of the GFF from the Brownian motion $(B_s)_{s\geq0}$ to $(B_s-Q\log(1+s))_{s\geq0}$. Up to time $t$, this corresponds to a change of measure given by the exponential martingale $e^{-Q\int_0^t\frac{dB_s}{1+s}-\frac{Q^2}{2}\int_0^t\frac{1}{(1+s)^2}ds}$, which is uniformly integrable since $\int_0^\infty\frac{1}{(1+t)^2}dt<\infty$. So the new field is absolutely continuous with respect to the old one, meaning that GMC does not make a difference between a Euclidean cylinder and a hyperbolic cusp.

Another way to see this is to look at the curvature, which reads in the distributional sense
\[K_{\hat{g}}=e^{-\frac{2\alpha}{Q}G(z,\cdot)}\left(K_g+\frac{4\pi\alpha}{Q}\left(\delta_z-\frac{1}{\Vol_g(\S^2)}\right)\right)\]
where $\Vol_g(\S^2)$ is the volume of the sphere in the metric $g$. Thus the metric has an atom of curvature at $z$, meaning it has a conical singularity.

Of course, when $\alpha=Q$, the singularity is no longer integrable and the metric looks like a semi-infinite (flat) cylinder near 0.

\section{The normalising constant in \eqref{eq:prob_density} and \eqref{eq:normalise_boundary}}
\label{app:constant}
We present the computation of the normalising constant for $f_{\gamma,\mu}$ (in a more general setting). The idea is that integrating over the location of a $\gamma$-insertion is the same as differentiating with respect to the cosmological constant. We present the main steps and leave the details to the reader. 

Let $N\geq3$ and $z_1,...,z_N\in\widehat{\C}$ pairwise disjoint and $(\alpha_1,...,\alpha_N)$ satisfying the Seiberg bounds. For notational convenience, we write $\mc{G}(x):=\sum_{i=1}^N\alpha_iG(z_i,x)$ and as usual $\sigma=\sum_{i=1}^N\frac{\alpha_i}{Q}-2$.

Using Cameron-Martin's theorem to go from the second to third line we find
\begin{equation}
\begin{aligned}
&\frac{1}{2}e^{-\underset{1\leq i<j}{\sum}\alpha_i\alpha_j G(z_i,z_j)}\int_{\widehat{\C}}\left\langle V_\gamma(z)\prod_{i=1}^NV_{\alpha_i}(z_i)\right\rangle dz\\
&\qquad=\int_{\widehat{\C}}e^{\gamma\mc{G}(z)}\int_\R e^{(Q(\sigma+\frac{\gamma}{Q})c}\E\left[\exp\left(-\mu e^{\gamma c}M^\gamma\left(e^{\gamma(\mc{G}+\gamma G(z,\cdot))}\right)\right)\right]dcd^2z\\
&\qquad=\E\left[\int_\R e^{Q\sigma c}e^{\gamma c}M^\gamma\left(e^{\gamma\mc{G}}\right)\exp\left(-\mu e^{\gamma c}M^\gamma\left(e^{\gamma\mc{G}}\right)\right)dc\right]\\
&\qquad=-\frac{1}{2}e^{-\underset{1\leq i<j}{\sum}\alpha_i\alpha_j G(z_i,z_j)}\frac{\partial}{\partial\mu}\left\langle\prod_{i=1}^NV_{\alpha_i}(z_i)\right\rangle
\end{aligned}
\end{equation}
so that in the end
\begin{equation}
\label{eq:constant}
\int_{\widehat{\C}}\left\langle V_\gamma(z)\prod_{i=1}^NV_{\alpha_i}(z_i)\right\rangle d^2z=-\frac{\partial}{\partial\mu}\left\langle\prod_{i=1}^NV_{\alpha_i}(z_i)\right\rangle=\frac{Q\sigma}{\gamma\mu}\left\langle\prod_{i=1}^NV_{\alpha_i}(z_i)\right\rangle
\end{equation}
where we simply used that $\langle\prod_{i=1}^NV_{\alpha_i}(z_i)\rangle$ is equal to $\mu^{-\frac{Q\sigma}{\gamma}}$ times some quantity independent of $\mu$. In particular this yields \eqref{eq:prob_density} for $N=3$ and $(\alpha_1,\alpha_2,\alpha_3)=(\gamma,\gamma,\gamma)$.

Similarly, in the disc case, we find that for $(\alpha_1,...,\alpha_N,\beta_1,...,\beta_M)$ satisfying the Seiberg bounds, we have
\[\int_\H\left\langle V_\gamma(z)\prod_{i=1}^NV_{\alpha_i}(z_i)\prod_{j=1}^MB_{\beta_j}(x_j)\right\rangle d^2z=-\frac{\partial}{\partial\mu}\left\langle\prod_{i=1}^NV_{\alpha_i}(z_i)B_{\beta_j}(x_j)\right\rangle\]
and
\[\int_\R\left\langle B_\gamma(x)\prod_{i=1}^NV_{\alpha_i}(z_i)\prod_{j=1}^MB_{\beta_j}(x_j)\right\rangle dx=-\frac{\partial}{\partial\mu_{\partial}}\left\langle\prod_{i=1}^NV_{\alpha_i}(z_i)B_{\beta_j}(x_j)\right\rangle\]
In general, this does not simplify as nicely as \eqref{eq:constant} but if e.g. $\mu=0$, then we have for instance
\[\int_\R\langle B_\gamma(x)V_\gamma(i)B_\gamma(0)\rangle dx=\frac{3\gamma-2Q}{2\gamma\mu}R(\gamma,\gamma)\]
\hspace{10 cm}


\begin{thebibliography}{20}
\bibitem[AS03]{angel} O. Angel, O. Schramm, \emph{Uniform Infinite Planar Triangulations}, Comm. Math.
Phys., 241(2-3):191–213, 2003.

\bibitem[Bav18]{Bav} G. Baverez, \emph{Modular bootstrap agrees with path integral in the large moduli limit} (2018), preprint \href{http://arxiv.org/abs/1805.09766}{arXiv:1805.09766}

\bibitem[Ber16]{Beb} N. Berestycki, \emph{Introduction to the Gaussian Free Field and liouville Quantum Gravity} (2016)

\bibitem[Ber17]{Be} N. Berestycki, \emph{An elementary approach to Gaussian multiplicative chaos}
Electr. Comm. Probab., vol. 22 (2017), no.27, 1-12.

\bibitem[BPZ84]{BPZ} A. A. Belavin, A. M. Polyakov, A. B. Zamolodchikov, \emph{Infinite conformal symmetry in two-dimensional quantum field theory},
Nuclear Physics B 241 (2), 333-380 (1984).

\bibitem[BZ06]{BZ} A. A. Belavin, Al. B. Zamolodchikov, \emph{Integrals over moduli space, ground ring, and four-point function in minimal Liouville gravity}, \emph{Theoretical and Mathematical Physics}, \textbf{147}(3): 729-754 (2006)

\bibitem[Dav88]{Da}F. David, \emph{Conformal Field Theories Coupled to 2-D Gravity in the Conformal Gauge}, \emph{Mod. Phys.
Lett. A 3} (1988) 1651-1656

\bibitem[DK89]{DiKa} J. Distler, H. Kawai, \emph{Conformal field theory and 2D quantum gravity}, \emph{Nuclear Physics B} (1989), Volume 321, pp 509-527

\bibitem[DO94]{DO} H. Dorn, H.-J. Otto, \emph{Two and three point functions in Liouville theory}, Nuclear Physics B 429 (2), 375-388 (1994)

\bibitem[DP86]{DhPh} E. D’Hoker, D.H. Phong, \emph{Multiloop amplitudes for the bosonic Polyakov string}, Nuclear Phys.B269(1986), no. 1, 205–234

\bibitem[DMS14]{DMS} B. Duplantier, J. Miller, S. Sheffield, \emph{Liouville quantum gravity as mating of trees} (2014) e-print \href{https://arxiv.org/abs/1409.7055}{arXiv:1409.7055}

\bibitem[DKRV15]{DKRV2} F. David, A. Kupiainen, R. Rhodes, V. Vargas, \emph{Renormalizability of Liouville quantum field theory at the Seiberg bound}, e-print \href{https://arxiv.org/abs/1506.01968v1}{arXiv:1506.01968v1}.

\bibitem[DKRV16]{DKRV1} F. David, A. Kupiainen, R. Rhodes, V. Vargas, \emph{Liouville Quantum Gravity on the Riemann sphere}, \emph{Commun. Math. Phys.}, March 2016, Volume 342, Issue 3, pp 869-907.

\bibitem[DRSV12]{DRSV} B. Duplantier, R. Rhodes, S. Sheffield, V. Vargas, \emph{Critical Gaussian Multiplicative Chaos: Convergence of the derivative martingale} (2012) preprint \href{https://arxiv.org/abs/1206.1671v3}{arXiv:1206.1671v3}

\bibitem[DRV16]{DRV} F. David, R. Rhodes, V. Vargas, \emph{Liouville Quantum Gravity on complex tori}, \emph{Journal of Mathematical Physics}, \textbf{57}, 022302 (2016).

\bibitem[Dub09]{Du} J. Dub\'edat, \emph{SLE and the free field: partition functions and couplings}, Journal
of the American Mathematical Society, 22(4):995–1054, 2009. 

\bibitem[Dub15]{dubedat} J. Dub\'edat, \emph{SLE and Virasoro representations: localization} Commun. Math. Phys. (2015) 336: 761. \href{https://doi.org/10.1007/s00220-014-2283-7}{doi:10.1007/s00220-014-2283-7}

\bibitem[Fri04]{friedrich} R. Friedrich, \emph{On Connections of Conformal Field Theory and
Stochastic Loewner Evolution} (2004) preprint \href{https://arxiv.org/abs/math-ph/0410029}{arXiv:math-ph/0410029}

\bibitem[FVV00]{fateev} V. Fateev, A. Zamolodchikov, Al. Zamolodchikov, \emph{Boundary Liouville Field Theory I. Boundary State and Boundary Two-point Function} (2000), preprint \href{https://arxiv.org/abs/hep-th/0001012}{arXiv:hep-th/0001012}

\bibitem[Gaw99]{Gaw} K. Gawedzki, \emph{Lectures on conformal field theory. In Quantum fields and strings: A course for mathematicians}, Vols. 1, 2 (Princeton, NJ, 1996/1997), pages 727-805. Amer. Math. Soc., Providence, RI, 1999

\bibitem[GRV16]{GRV}C. Guillarmou, R. Rhodes, V. Vargas, \emph{Polyakov's formulation of $2d$ bosonic string theory} (2017), preprint \href{https://arxiv.org/abs/1607.08467v3}{arXiv:1607.08467v3}.

\bibitem[HJS10]{HJS} L. Hadasz, Z. Jask\'olski and P. Suchanek, \emph{Modular bootstrap in Liouville field theory}, Phys. Lett. B685 (2010), 79–85.

\bibitem[HMW11]{HMW} D. Harlow, J. Maltz, E. Witten: \emph{Analytic Continuation of Liouville Theory}, Journal of High Energy Physics (2011).

\bibitem[HRV15]{huang} Y. Huang, R. Rhodes, V. Vargas, \emph{Liouville Quantum Gravity on the unit disk} (2017) preprint \href{https://arxiv.org/abs/1502.04343}{arXiv:1502.04343}

\bibitem[Kup16]{Ku} A. Kupiainen, \emph{Constructive Liouville Conformal Field Theory} (2016), preprint \href{https://arxiv.org/abs/1611.05243}{arXiv:1611.05243}

\bibitem[KMSW16]{KMSW} R. Kenyon, J. Miller, S. Sheffield, D. B. Wilson, \emph{Bipolar orientations on maps and SLE$_{12}$} (2016) preprint \href{https://arxiv.org/abs/1511.04068}{arXiv:1511.04068}

\bibitem[KPZ88]{KPZ} V.G. Knizhnik, A.M. Polyakov, A.B. Zamolodchikov, \emph{Fractal structure of 2D-quantum gravity}, Modern
Phys. Lett A, 3(8) (1988), 819-826.

\bibitem[KRV15]{KRV1} A. Kupiainen, R. Rhodes, V. Vargas, \emph{Local conformal structure of Liouville Quantum Gravity} (2015), e-print \href{https://arxiv.org/abs/1512.01802}{arXiv:1512.01802}

\bibitem[KRV17]{KRV2} A. Kupiainen, R. Rhodes, V. Vargas, \emph{Integrability of Liouville theory: Proof of the DOZZ formula} (2017), e-print \href{https://arxiv.org/abs/1707.08785v1}{arXiv:1707.08785}

\bibitem[MY16]{MY} B. Mallein, M. Yor, \emph{Exercices sur les temps locaux de semi-martingales continues et les excursions browniennes} (2016), \href{https://arxiv.org/pdf/1606.07118}{arXiv:1606.07118}

\bibitem[OPS88]{OPS} B. Osgood, R. Phillips, and P. Sarnak, \emph{Extremals of Determinantsof Laplacians}, \emph{Journal of Functional Analysis} \textbf{80}, 148-211 (1988)

\bibitem[Pol81]{Po} A.M. Polyakov, \emph{Quantum geometry of bosonic strings}, Phys. Lett. 103B 207 (1981)

\bibitem[PT01]{ponsot} B. Ponsot, J. Teschner, \emph{Boundary Liouville field theory: Boundary three-point function}, Nuclear Physics B 622 (2002) 309–327


\bibitem[Rib14]{Ri} S. Ribault, \emph{Conformal Field theory on the plane}, e-print \href{https://arxiv.org/abs/1406.4290}{arXiv:1406.4290}

\bibitem[RV13]{RVb} R. Rhodes, V. Vargas, \emph{Gaussian multiplicative chaos and applications: a review} (2013) e-print \href{https://arxiv.org/abs/1305.6221}{arXiv:1305.6221}

\bibitem[RV16]{RhVa} R. Rhodes, V. Vargas, \emph{Lecture notes on Gaussian multiplicative chaos and Liouville Quantum Gravity} (2016) preprint \href{https://arxiv.org/abs/1602.07323}{arXiv:1602.07323}

\bibitem[Sei90]{Sei} N. Seiberg, \emph{Notes on Quantum Liouville Theory and Quantum Gravity}, \emph{Progress of Theoretical Physics}, supp. 102, 1990

\bibitem[She10]{Sh} S. Sheffield. \emph{Conformal weldings of random surfaces: SLE and the quantum
gravity zipper} (2010), e-print \href{https://arxiv.org/abs/1012.4797}{arXiv:1012.4797}

\bibitem[Tes95]{T95} J. Teschner \emph{On the Liouville three point function}, Phys. Lett. B 363, 65-70 (1995)

\bibitem[TV15]{TV} J. Teschner, G. S. Vartanov, \emph{Supersymmetric gauge theories, quantization of moduli spaces of
flat connections, and conformal field theory}. Adv. Theor. Math. Phys. 19 (2015) 1–135

\bibitem[Tut63]{Tutte} W. T. Tutte, \emph{A census of planar maps}, Canadian Journal of Mathematics, 15, 249-271 (1963).


\bibitem[Wil74]{Wil} D. Williams, \emph{Path decomposition and continuity of local times for one-dimensional diffusions, I}.
Proceed. of London Math. Soc. (3), 28, 738-768, (1974).

\bibitem[ZZ96]{ZZ} A. B. Zamolodchikov, Al. B. Zamolodchikov, \emph{Structure constants and conformal bootstrap in Liouville field theory}, Nuclear
Physics B 477 (2), 577-605 (1996)



\end{thebibliography}
 \end{document}